\documentclass[11pt,twoside,leqno]{article}

\usepackage{amssymb}
\usepackage{amsmath}
\usepackage{mathrsfs}
\usepackage{amsthm}
\usepackage{txfonts}
\usepackage{color}
\usepackage[colorlinks,linkcolor=red]{hyperref}
\usepackage{citeref}
\usepackage{indentfirst}
\usepackage{indentfirst}

\usepackage{anysize}

\usepackage[final,notref,notcite]{showkeys}

\allowdisplaybreaks

\def\ls{\lesssim}

\def\fz{\infty}

\renewcommand{\r}{\right}
\newcommand{\lf}{\left}

\def\ls{\lesssim}

\def\rr{{\mathbb R}}

\def\rn{{{\rr}^n}}

\def\az{\alpha}

\def\epz{\varepsilon}

\def\bz{\beta}

\def\gz{{\gamma}}

\def\ls{\lesssim}

\def\boz{\Omega}

\newcommand{\R}{{\mathbb R}}

\newcommand{\Z}{{\mathbb Z}}

\newtheorem{theorem}{Theorem}[section]
\newtheorem{lemma}[theorem]{Lemma}
\newtheorem{corollary}[theorem]{Corollary}
\newtheorem{proposition}[theorem]{Proposition}
\theoremstyle{definition}
\newtheorem{remark}[theorem]{Remark}

\newtheorem{definition}[theorem]{Definition}

\def\diam{{\mathop\mathrm{diam\,}}}
\def\dist{{\mathop\mathrm{\,dist\,}}}

\numberwithin{equation}{section}

\newcommand{\G}{\mathcal G_s}
\newcommand{\dd}{\,\mathrm d}
\newcommand{\one}{\mathbf 1}

\begin{document}
\title{\Large\bf Optimal Global Sobolev Estimates for Fractional Dirichlet Problems in Bounded Lipschitz Domains
\footnotetext{\hspace{-0.35cm} 2020 {\it Mathematics Subject
Classification}. {Primary 35R11; Secondary 35B65, 46E35.}
\endgraf{\it Key words and phrases}. Fractional Laplacian, Lipschitz domain, Green potential, reverse H\"older estimate, Sobolev regularity.
\endgraf This project is partially supported by the National Natural Science Foundation of China (Grant No. 12431006) and Longyuan Young Talents of Gansu Province.
}}
\author{Wenxian Ma\,\footnote{Corresponding author, E-mail:
\texttt{mawenxian@lzu.edu.cn}}\, and Sibei Yang}
\date{}
\maketitle

\vspace{-0.8cm}

\begin{center}
\begin{minipage}{13.5cm}\small
{{\bf Abstract.} Let $n\geq2$, $s\in(0,1)$, and let $\Omega\subset\mathbb{R}^n$
be a bounded Lipschitz domain. In this paper, we establish optimal global Sobolev
estimates for the fractional Dirichlet problem
\begin{equation*}
\left\{\begin{aligned}
(-\Delta)^su&=f & & \text{in}\ \ \Omega,\\
 u&=0 & & \text{in}\ \ \mathbb{R}^n\setminus\Omega.
\end{aligned}\right.
\end{equation*}
More precisely, for any given $t\in[s,\min\{2s,\,s+1/2\})$, we prove
that there exists a small positive constant $\eta=\eta(n,s,t,\Omega)$
such that the following holds. If $n=2$ and $s\in(1/2,1)$,
then, for any $q\in(1,\frac{4}{1+2s}+\eta)$ and $p\in(1,\frac{nq}{n-(2s-t)q}]$,
\begin{equation*}
\|u\|_{W^{t,p}(\mathbb{R}^n)}\le C\|f\|_{L^q(\Omega)},
\end{equation*}
and, if either $n=2$ and $s\in(0,1/2]$ or $n\ge 3$ and $s\in(0,1)$,
then, for any $q\in(1,\frac{n(2s+1)}{nt+(2s-t)(2s+1)}+\eta)$ and
$p\in(1,\frac{nq}{n-(2s-t)q}]$,
\begin{equation*}
\|u\|_{W^{t,p}(\mathbb{R}^n)}\le C\|f\|_{L^q(\Omega)}.
\end{equation*}
Here the positive constant $C$ depends only on $n$, $s$, $t$, $p$, $q$, and $\Omega$. The stated universal baselines are optimal in Dahlberg's sense:
for any larger target exponent $p$ there is a bounded Lipschitz domain whose solution with right-hand side identically $1$ is not in $W^{t,p}(\mathbb{R}^n)$.
The counterexamples are constructed using regular Cantor sets on the boundary and, in the planar exceptional case, homogeneous solutions in sectors.
}
\end{minipage}
\end{center}

\vspace{0.1cm}

\tableofcontents

\section{Introduction}

Let $n\geq2$, $s\in(0,1)$, and let $\Omega\subset\mathbb{R}^n$
be a bounded Lipschitz domain.
In this paper, we consider optimal global Sobolev
estimates for the fractional Dirichlet problem
\begin{equation}\label{eq1.1}
\left\{\begin{aligned}
(-\Delta)^su&=f & & \text{in}\ \ \Omega,\\
 u&=0 & & \text{in}\ \ \mathbb{R}^n\setminus\Omega,
\end{aligned}\right.
\end{equation}
Recall that, for any $u\in C_{\rm c}^\infty(\mathbb{R}^n)$, the
fractional Laplacian $(-\Delta)^su$ is defined by setting, for any $x\in\mathbb{R}^n$,
\begin{equation}\label{eq1.2}
(-\Delta)^su(x):=c_{(n,s)}\,\mathrm{P.V.}
 \int_{\mathbb{R}^n}\frac{u(x)-u(y)}{|x-y|^{n+2s}}dy,
\end{equation}
where
$$c_{(n,s)}:=\frac{s\,2^{2s}\Gamma(n/2+s)}{\pi^{n/2}\Gamma(1-s)},$$
$\Gamma$ denotes the Gamma function, and $\mathrm{P.V.}$ denotes
the Cauchy principal value (see, for instance, \cite[Section 2.4]{K17}).
Let $f\in L^1(\Omega)$. A function $u\in L^1(\mathbb{R}^n)$ is called a
\emph{weak solution} of problem \eqref{eq1.1} if $u=0$ almost everywhere
in $\Omega^c:=\mathbb{R}^n\setminus\Omega$ and, for any function $\phi\in C_{\rm c}^\infty(\Omega)$,
\begin{equation}\label{eq1.3}
 \int_{\mathbb{R}^n}u(x)(-\Delta)^s\phi(x)\,dx
 =\int_\Omega f(x)\phi(x)\,dx.
\end{equation}

Let $E\subset\mathbb{R}^n$ be a domain and $p\in[1,\infty)$. For any given $t\in(0,1)$,
the fractional Sobolev space $W^{t,p}(E)$ is defined by
$$
W^{t,p}(E):=\{v\in L^p(E):[v]_{W^{t,p}(E)}<\infty\},
$$
equipped with the norm
$$
\|v\|_{W^{t,p}(E)}:=\|v\|_{L^p(E)}+[v]_{W^{t,p}(E)},
$$
where
\begin{equation*}
[v]_{W^{t,p}(E)}:=\left(\int_E\int_E\frac{|v(x)-v(y)|^p}
{|x-y|^{n+tp}}dydx\right)^{\frac1p}.
\end{equation*}
For $t=1$, $W^{1,p}(E)$ consists of the functions in $L^p(E)$ whose
distributional gradient belongs to $L^p(E;\mathbb{R}^n)$, with norm
$\|v\|_{L^p(E)}+\|\nabla v\|_{L^p(E;\mathbb{R}^n)}$.
For $t\in(1,2)$, let
$$W^{t,p}(E):=\{v\in W^{1,p}(E):[\nabla v]_{W^{t-1,p}(E)}<\infty\}
$$
and
$$
\|v\|_{W^{t,p}(E)}:=\|v\|_{W^{1,p}(E)}+[\nabla v]_{W^{t-1,p}(E;\mathbb{R}^n)}.
$$
Moreover, we use the tail weight $w_s:=(1+|\cdot|)^{-n-2s}$
and the associated weighted space
\begin{equation*}
L^1(\mathbb{R}^n,w_s):=\left\{v:\int_{\mathbb{R}^n}|v(x)|w_s(x)\,dx<\infty\right\}.
\end{equation*}
We write $\|v\|_{L^1(w_s)}:=\int_{\mathbb{R}^n}|v(x)|w_s(x)\,dx$.
For $v$ in this space, $(-\Delta)^sv$ is well defined as a distribution
by, for any $\phi\in C_{\rm c}^\infty(\mathbb{R}^n)$, $\langle(-\Delta)^sv,\phi\rangle=\int_{\mathbb{R}^n}v(-\Delta)^s\phi\,dx$.
Indeed, $(-\Delta)^s\phi$ is bounded and decays as
$(1+|x|)^{-n-2s}$ at infinity.
Furthermore, a function $v\in L^1(\mathbb{R}^n,w_s)$ is called \emph{$s$-harmonic} in an open set $O$ if $(-\Delta)^sv=0$ there in this distributional sense (see Definition~\ref{def:sharm}).
An \emph{energy solution} is a weak solution that additionally belongs to $W^{s,2}(\mathbb{R}^n)$.

Let $p\in(1,\infty)$ and $p':=p/(p-1)$. For any $v\in W^{s,p}(\mathbb{R}^n)$
and $\phi\in W^{s,p'}(\mathbb{R}^n)$, define
\begin{equation*}
\mathcal E_s(v,\phi)
:=\frac{c_{(n,s)}}2\iint_{\mathbb{R}^n\times\mathbb{R}^n}
 \frac{[v(x)-v(y)][\phi(x)-\phi(y)]}{|x-y|^{n+2s}}\,dy\,dx.
\end{equation*}
By H\"older's inequality, we find that this integral is absolutely convergent and
$$
|\mathcal E_s(v,\phi)|\leq\frac{c_{(n,s)}}2
 [v]_{W^{s,p}(\mathbb{R}^n)}[\phi]_{W^{s,p'}(\mathbb{R}^n)}.
$$
\begin{remark}\label{re1.1}
If $u\in W^{s,p}(\R^n)$ vanishes outside $\Omega$, then
\eqref{eq1.3} is equivalent to
$\mathcal E_s(u,\phi)=\int_\Omega f\phi\,dx$ for
$\phi\in C_{\rm c}^\infty(\Omega)$. Indeed, this follows by symmetrizing the integral over $|x-y|>a$ and letting $a\to0^+$. The double integral converges absolutely, and the truncated fractional Laplacian of $\phi$ converges uniformly. In particular, for energy solutions this is the usual variational identity for compactly supported smooth test functions.
Whenever the datum defines a continuous functional on
$$
V_s(\Omega):=\{v\in W^{s,2}(\mathbb{R}^n):v=0\text{ a.e. in }\Omega^c\},
$$
the identity extends by density to all test functions in $V_s(\Omega)$.
\end{remark}

\subsection{Main results}

In this paper, motivated by recent works of Borthagaray and Nochetto \cite{BN2023,bn23}, Abdellaoui et al. \cite{AFLY25}, and
Colombo et al. \cite{CFR26} on global Sobolev estimates for weak solutions
to problem \eqref{eq1.1}, we study optimal global fractional Sobolev estimates for such solutions in bounded Lipschitz domains.

Recall that, for the classical Laplacian, Dahlberg \cite[Theorem~1]{Da79} established optimal global Sobolev estimates on the exact first-order Sobolev gain line, using gradient estimates for Green potentials in Lipschitz domains, with a range of data exponents depending on the domain.

For the fractional setting, motivated by the work of Colombo et al. \cite{CFR26}, we introduce the following volume reverse H\"older exponent.

\begin{definition}\label{d1.1}
Let $r_0\in(0,\infty)$ be the Lipschitz localization radius
as in Definition~\ref{d2.1}, and let
$D_r(\xi):=B(\xi,r)\cap\Omega$ for any $\xi\in\partial\Omega$ and $r\in(0,r_0)$. A function $u$ is called \emph{admissible on the boundary
patch $B(\xi,r)$} if $u\geq0\text{ in }\mathbb{R}^n$,
$$
 u\in C(B(\xi,r))\cap L^1(\mathbb{R}^n,w_s),
$$
and
\begin{equation*}
\left\{\begin{aligned}
(-\Delta)^su&=0 &&\text{in }D_r(\xi),\\ \nonumber
u&=0 &&\text{in }B(\xi,r)\cap\Omega^c,
\end{aligned}\right.
\end{equation*}
where the harmonicity is understood distributionally and the continuous
representative is used on $B(\xi,r)$.
The \emph{volume reverse H\"older exponent} $p_\Omega^{\rm vol}$
is defined as the supremum of the exponents $p\in(1,\infty)$ for which
\begin{equation}\label{eq1.4}
\left\{\fint_{D_{r/8}(\xi)}
\left[\frac{u(x)}{\delta(x)^s}\right]^p\,dx\right\}^{\frac1p}
\leq C\fint_{D_{r/8}(\xi)}\frac{u(x)}{\delta(x)^s}\,dx
\end{equation}
holds for any $\xi\in\partial\Omega$, any $r\in(0,r_0)$, and any
admissible $u$, where the positive constant $C$ depends only on $n,s,p$, and $\Omega$. Here, for a measurable set $E\subset\mathbb{R}^n$, $\fint_E \cdot\,dx:=|E|^{-1}\int_E \cdot\,dx$.
The radius $r_0$ is chosen small enough, depending only on the domain geometry, so that all boundary estimates below hold.
\end{definition}

Recall that the Dirichlet Green function $G_s(\cdot,\,\cdot)$ is the positive symmetric kernel associated with the zero-exterior operator. For a fixed
$y\in\Omega$, it is defined by
\begin{equation}
 \left\{\begin{aligned}
 (-\Delta)^s_xG_s(x,y)&=\delta_y &&\text{in}\quad\Omega,\\
 G_s(x,y)&=0 &&\text{in}\quad\mathbb{R}^n\setminus\Omega,
 \end{aligned}\right.
\end{equation}
interpreted distributionally, where $\delta_y$ is the
Dirac mass at $y$. Equivalently, for any $\phi\in C_{\rm c}^\infty(\Omega)$,
\[
 \int_{\mathbb{R}^n}G_s(x,y)(-\Delta)^s\phi(x)\,dx=\phi(y).
\]
For any given $f\in L^1(\Omega)$, the \emph{Green potential} $\G f$ is defined
by setting, for any $x\in\mathbb{R}^n$,
\[
 \G f(x)=\int_\Omega G_s(x,y)f(y)\,dy.
\]
Then $\G f\in L^1(\mathbb{R}^n)$ and satisfies \eqref{eq1.3}
(see, for example, \cite{B97,B99,cv14,k13,lz25}).

We now state the main results of this paper.

\begin{theorem}[Fractional Dahlberg type Sobolev estimate]\label{t1.1}
Let $n\ge2$, $s\in(0,1)$, and let $\Omega\subset\mathbb{R}^n$ be a bounded
Lipschitz domain. Assume that $p\in(1,p_\Omega^{\rm vol})$ and
$q\in(1,\infty)$ satisfy
\begin{equation}\label{eq1.5}
 \frac1p\ge\frac1q-\frac sn.
\end{equation}
Then, for any given $f\in L^q(\Omega)$, the Green potential $u:=\G f$
is a weak solution of \eqref{eq1.1} in the sense of
\eqref{eq1.3} and satisfies
\begin{equation}\label{eq1.9}
 \|u\|_{W^{s,p}(\R^n)}\le C\|f\|_{L^q(\Omega)}.
\end{equation}
Here $p_\Omega^{\rm vol}$ is defined in Definition~\ref{d1.1},
and the positive constant $C$ depends only on $n,\,s,\,p,\,q$, and $\Omega$.
\end{theorem}

Let $\beta_\Omega\in(0,s)$ be an upper boundary decay exponent as in Lemma~\ref{le2.9}; that is, near the boundary, every admissible function is comparable to the $\beta_\Omega$-th power of the distance to the boundary. In dimension two with $s\in(1/2,1)$, it can be chosen so that $\beta_\Omega>s-1/2$. Define, for any $t\in[s,\,\min\{s+1/2,2s\})$,
\begin{equation}\label{eq:Pt}
 P_{t,\Omega}:=\frac{1+2(s-\beta_\Omega)}{t-\beta_\Omega},
 \qquad
 P_0(n,s,t):=
 \begin{cases}
 \displaystyle\frac4{1+2(t-s)},&n=2\ \text{and}\ s\in(1/2,1),\\[4pt]
 \displaystyle\frac{2s+1}{t},&\text{otherwise}.
 \end{cases}
\end{equation}
Then $P_{t,\Omega}>P_0(n,s,t)>2$.

\begin{proposition}[The volume reverse H\"older exponent]\label{prop:RH}
Let $n\ge2$, $s\in(0,1)$, and let $\Omega\subset\mathbb{R}^n$ be a bounded
Lipschitz domain. Then, for any $p\in(1,\,P_{s,\Omega})$, \eqref{eq1.4}
holds. Consequently,
\begin{equation}
 p_\Omega^{\rm vol}\ge p_*:=P_{s,\Omega}
 =2+\frac1{s-\beta_\Omega}>P_0(n,s,s).
\end{equation}
In particular, $p_\Omega^{\rm vol}>4$ if $n=2$ and $s\in(1/2,1)$.
\end{proposition}

We also consider the higher-order Sobolev estimates for the weak solution to
problem \eqref{eq1.1}.

\begin{theorem}[Higher-order Sobolev estimates]\label{thm:higher}
Let $n\ge2$, $s\in(0,1)$, and let $\Omega\subset\mathbb{R}^n$ be a bounded
Lipschitz domain. Assume that $t\in[s,\min\{s+1/2,2s\})$,
$p\in(1,P_{t,\Omega})$, and $q\in(1,\infty)$ satisfy
\begin{equation}\label{e1.9}
 \frac1p\ge\frac1q-\frac{2s-t}{n}.
\end{equation}
Then, for any given $f\in L^q(\Omega)$, the Green potential $u=\G f$
is a weak solution of \eqref{eq1.1} and satisfies
\begin{equation}\label{e1.10}
\|u\|_{W^{t,p}(\R^n)}\le C\|f\|_{L^q(\Omega)},
\end{equation}
where the positive constant $C$ depends only on $n,\,s,\,t,\,p,\,q$, and $\Omega$.
\end{theorem}

\begin{remark}
The gain relation \eqref{e1.9} is forced by the interior
scaling. For a smooth function $\psi$, define the bump by $\psi_r(x):=\psi((x-x_0)/r)$ supported
inside the domain, the homogeneous Sobolev seminorm scales as
\[
 \|\psi_r\|_{\dot W^{t,p}}=r^{n/p-t}\|\psi\|_{\dot W^{t,p}}.
\]
The corresponding data satisfy
\[
 \|(-\Delta)^s\psi_r\|_{L^q(\Omega)}\le Cr^{n/q-2s}.
\]
For $t=1$, the homogeneous norm means $\|\nabla\psi_r\|_p$.
Thus an estimate with a constant independent of $r$ requires
\eqref{e1.9}. 
\end{remark}

On the exact gain line the corresponding universal data exponent is
\begin{equation*}
 Q_0(n,s,t):=\frac{nP_0(n,s,t)}{n+(2s-t)P_0(n,s,t)}
 =\begin{cases}
 \displaystyle\frac4{1+2s},&n=2 \ \text{and}\ s\in(1/2,1),\\[4pt]
 \displaystyle\frac{n(2s+1)}{nt+(2s-t)(2s+1)},&\text{otherwise}.
 \end{cases}
\end{equation*}

\begin{corollary}\label{c1.1}
Let $n\ge2$, $s\in(0,1)$, and let $\Omega\subset\R^n$ be a bounded
Lipschitz domain. Assume that $t\in[s,\min\{s+1/2,2s\})$. Then there is a constant $\eta\in(0,\infty)$, depending only on $n,\,s,\,t$, and $\Omega$,
such that $Q_0(n,s,t)+\eta<n/(2s-t)$ and,
for any $q\in(1,Q_0(n,s,t)+\eta)$ and
$p\in(1,\frac{nq}{n-(2s-t)q}],$ \eqref{e1.10} holds.
In particular, at $t=s$, one has \eqref{eq1.9} holds.
\end{corollary}

\begin{remark}
In \cite{BN2023}, Borthagaray and Nochetto combined localized differences with the variational structure to establish Besov regularity for the weak solution to \eqref{eq1.1} in bounded Lipschitz domains for $p=2$.
However, their approach does not extend directly to general $p\neq2$.
Colombo et al. \cite{CFR26} obtained the same regularity results by establishing the $L^2$ solvability of the homogeneous version of \eqref{eq1.1}. Corollary \ref{c1.1} extends their result to general $p\neq2$, which is shown to be optimal in Theorem \ref{thm:optimal}.

Moreover, in \cite{FMY2025}, Fu et al. proved that, on a bounded Lipschitz domain, there exists $\epsilon\in(0,s)$
such that, for any $t\in[\epsilon,\min\{1+\epsilon,2s\})$, $q\in(n/(2s-t),\infty]$, and $p\in[1,1/(t-\epsilon))$,
\begin{equation}\label{eq1.6}
\|u\|_{W^{t,p}(\mathbb R^n)}\leq C\|f\|_{L^q(\Omega)}.
\end{equation}
At $t=s$, \eqref{eq1.6} requires $q>n/s$. The purpose of this paper is to
obtain the exact gain line for less integrable data, for $t=s$ and for the higher orders in Theorem~\ref{thm:higher}.
\end{remark}

\begin{theorem}[Optimality of the universal range]\label{thm:optimal}
Let $n\ge2$, $s\in(0,1)$, and $t\in[s,\min\{s+1/2,2s\})$. Then,
for any given $p\in(P_0(n,s,t),\infty)$, there is a bounded Lipschitz
domain $\Omega_p\subset\mathbb{R}^n$ such that the weak solution to
\begin{equation*}
 \left\{\begin{aligned}
 (-\Delta)^su&=1 &&\text{in}\quad\Omega_p,\\
 u&=0 &&\text{in}\quad\mathbb{R}^n\setminus\Omega_p
 \end{aligned}\right.
\end{equation*}
does not belong to $W^{t,p}(\mathbb{R}^n)$.
\end{theorem}

Theorem~\ref{thm:optimal} uses the meaning of optimality stated after
\cite[Theorem~1]{Da79} of Dahlberg. The domain, including
its Lipschitz constant, may depend on $p$, while $s$ and $t$ remain
fixed. Each fixed domain has the strict improvement in
Theorem~\ref{thm:higher}.
The quantity $Q_0$ concerns exponent pairs on the exact gain line,
not a restriction on larger data exponents with a smaller target $p$.
In dimension two with $s>1/2$, this data baseline is independent of $t$.

Homogeneous solutions in cones describe a complementary obstruction
to boundary regularity. Michalik \cite{M06} proved continuity
of the characteristic exponent with respect to the aperture and
identified its limiting value when the complement shrinks to a ray.
Terracini et al. \cite{TTV2018} studied the associated
$s$-harmonic profiles and their variational properties. We use these
results both in the planar upper decay estimate and in the sector
counterexample. In the remaining dimensions a single conical vertex
does not yield the required universal obstruction. A regular Cantor
set on a Lipschitz graph permits the necessary continuous choice
of the dimension of the singular set.

The following lifting result is of independent interest beyond its role in the boundary estimate.

\begin{proposition}[Sobolev lifting]\label{prop:lifting}
Let $n\ge2$, $s\in(0,1)$, $t\in[s,\min\{s+1/2,2s\})$, and let
$\Omega\subset\mathbb{R}^n$ be a bounded Lipschitz domain. Assume that
$1<q<p<\infty$ and $\gamma:=n(1/q-1/p)\in(0,\,2s-t]$.
If $f\in L^q(\Omega)$ and $\G|f|/\delta^t\in L^p(\Omega)$, then
\begin{equation}\label{eq:lifting}
\|\G f\|_{W^{t,p}(\mathbb{R}^n)}\le C\left[\|f\|_{L^q(\Omega)}
+\left\|\G|f|/\delta^t\right\|_{L^p(\Omega)}\right],
\end{equation}
where the constant $C$ depends only on $n,s,t,p,q$, and $\Omega$.
Moreover, for these exponents the uniform quotient estimate
\begin{equation}\label{eq:GQ}
\left\|\G|f|/\delta^t\right\|_{L^p(\Omega)}\le C\|f\|_{L^q(\Omega)}
\end{equation}
is equivalent to the uniform Sobolev estimate
\begin{equation}\label{eq:WS}
\|\G f\|_{W^{t,p}(\mathbb{R}^n)}\le C\|f\|_{L^q(\Omega)}.
\end{equation}
\end{proposition}

\subsection{Background and ideas of the proof}

The boundary behavior of solutions is a central issue in global
regularity for \eqref{eq1.1} (see, for example, \cite{AFLY25,ab17,bfr18,bfr2018,B97,
BN2023,bn23,CFR26,fr24,kw25,lz25,my24,RS2014,RS2017,rs13,s07}). Ros-Oton and Serra
\cite{RS2014} proved global $C^s$ regularity for bounded
data on bounded Lipschitz domains satisfying a uniform exterior ball
condition. Write $\delta(x)=\dist(x,\partial\Omega)$ for
$x\in\R^n$. Their quotient regularity theorem on $C^{1,1}$ domains
\cite[Theorem~1.2]{RS2014} identifies $\delta^s$ as the natural boundary
profile. For a wider class of stable operators, they subsequently \cite[Theorem~1.2 and Corollary~1.4]{RS2017} obtained
$u/\delta^s\in C^\kappa$ on $C^{1,\kappa}$ domains with $0<\kappa<s$,
and a boundary Harnack principle on $C^1$ domains. These results distinguish
regularity of an individual solution from regularity of its quotient by
the boundary profile. Kim and Weidner \cite[Theorems~1.1
and 1.6]{kw25} established boundary $C^s$ estimates and
Green-function estimates on $C^{1,\kappa}$ domains for symmetric
divergence-form kernels satisfying their ellipticity and H\"older
continuity assumptions.

The ball provides an explicit model for this boundary profile.
Dyda's power-function formula \cite{Dyda2012} gives, for $x\in B(0,1)$,
\[
 (-\Delta)^s(1-|x|^2)_+^s
 =\frac{2^{2s}\Gamma(s+1)\Gamma(n/2+s)}{\Gamma(n/2)},
\]
where, for any $a\in\R$, $a_+:=\max\{a,0\}$. The corresponding zero-exterior solution
is comparable to $\delta^s$ near the boundary of the ball. This model
explains the distance quotient, while the reverse H\"older estimate
used below accommodates the less uniform boundary behavior of
general Lipschitz domains.

Global Sobolev estimates have a different source of difficulty.
For bounded $C^2$ domains, Abdellaoui et al.
\cite{AFLY25} developed global estimates in Bessel-potential and fractional
Sobolev spaces using pointwise bounds for fractional derivatives of the
Green kernel. Ma and Yang \cite{my24} studied endpoint questions on bounded
$C^2$ domains, including bounded mean oscillation estimates for fractional derivatives
and the Riesz fractional gradient, as well as estimates for
Hardy-space data. As for non-smooth domains,
Acosta and Borthagaray \cite{ab17}
derived fractional and weighted Sobolev estimates on bounded Lipschitz
domains satisfying an exterior ball condition. Their estimates use
the H\"older regularity of the datum and lead to finite element error
bounds on quasi-uniform and graded meshes. This work exhibits how
boundary decay and interior H\"older estimates combine to control
global difference integrals.
On bounded Lipschitz domains, Borthagaray
and Nochetto \cite{BN2023} used localized differences and the
variational structure to prove Besov shift estimates for $L^2$ data. In
particular, for every sufficiently small $\eta>0$,
\[
 \|u\|_{W^{s+\min\{s,1/2\}-\eta,2}(\mathbb R^n)}
 \leq C_\eta\|f\|_{L^2(\Omega)}.
\]
The Besov endpoint when $s=1/2$ has a separate formulation in that work.
Borthagaray and Nochetto \cite{bn23} also considered the general Sobolev scale
in their work on graded meshes. Under the
additional global $C^s$ estimate for the solution and suitable
H\"older regularity of the datum, they gave
$W^{t,p}$ regularity subject to $t<s+1/p$ and the corresponding
interior differentiability restriction. The boundary integrability
condition arises from integrating powers of the distance function.
Fu et al. \cite{FMY2025} adapted this separation between
interior regularity and boundary growth to general Lipschitz domains,
with a domain-dependent H\"older exponent in place of $s$.

Weighted estimates offer another way to account for boundary
singularities. Dong and Ryu \cite{DR24} developed weighted
Sobolev estimates for elliptic and parabolic equations with general
symmetric stable operators on $C^{1,\kappa}$ open sets. Their
operator class allows singular L\'evy measures, and the parabolic
coefficients may be merely measurable in time. The weights encode
distance to the boundary. Here we treat the isotropic operator on
general bounded Lipschitz domains and recover an unweighted
whole-space $W^{t,p}$ norm from a single distance quotient.

Boundary comparison and Green function estimates provide the
potential-theoretic basis for this problem. Bogdan \cite{B97}
proved the boundary Harnack principle for nonnegative harmonic
functions of symmetric stable processes on Lipschitz domains.
In \cite{B99}, Bogdan established
the representation theory of fractional harmonic measure,
which further describes the fractional harmonic functions.
Bogdan et al. \cite{BKN2002} gave
gradient bounds for positive fractional harmonic functions.
Fu et al. \cite{FMY2025} obtained the
boundary decay and interior gradient estimates for the Green function
$G_s(\cdot,\cdot)$, which play an important role in
handling the fractional Dirichlet problem on Lipschitz domains.
Green potentials also play a central role for data outside the energy
duality range. On bounded $C^2$ domains, Chen and V\'eron
\cite{cv14} obtained weighted weak-Lebesgue estimates for
Green potentials of measures and used them to solve semilinear
fractional equations with subcritical absorption.

In this paper, we establish optimal global Sobolev estimates
for the fractional Dirichlet problem \eqref{eq1.1} in bounded Lipschitz domains.
Proposition~\ref{prop:RH} plays a central role in the proof.
The argument separates the boundary estimate for $\G|f|/\delta^s$
from the recovery of the full Sobolev norm.
For a nonnegative function harmonic near a boundary patch, boundary
Harnack principle transfers the Green-function square estimate on parallel
surfaces to $\G|f|/\delta^s$. The two inputs are recorded as
Lemmas~\ref{thm:bhp} and \ref{thm:square} in
Section~\ref{s2.1}.
The boundary H\"older estimate controls the remaining powers of the
quotient. Integrating over the distance level sets gives
\eqref{eq1.4} whenever $(s-\beta_\Omega)(p-2)<1$ for $p>2$. Meanwhile,
lower exponents follow by H\"older's inequality. This yields
Proposition~\ref{prop:RH}.

Next, a covering by boundary patches and interior regions converts
\eqref{eq1.4} into uniform annular moments of the kernel
$G_s(x,y)/\delta(x)^s$. A dual maximal-function argument then gives
the quotient estimate on the exact gain line $1/p=1/q-s/n$.
From the fact an annular exponent strictly larger than the target exponent,
it follows that a strong estimate holds on the exact gain line.

To pass from the quotient $\G|f|/\delta^s$ to $W^{t,p}(\R^n)$ norm of $\G f$, we decompose bounded
nonnegative data into a part supported near each Whitney cube and
its complement. The local Green potential is compared with the
whole-space Riesz potential; both harmonic remainders remain
nonnegative on $\R^n$. Difference estimates for the Riesz potential
and interior estimates for the harmonic terms control nearby point
pairs. The distance quotient controls distant interior pairs and
pairs crossing the boundary. Summation over the Whitney cubes proves
Proposition~\ref{prop:lifting}, including the strictly subcritical case and the gradient differences required when $t>1$.

For $t>s$, the quotient becomes $u/\delta^t$, and the coarea
condition is $(t-s)p+(s-\beta_\Omega)(p-2)<1$. This yields
$P_{t,\Omega}$ in \eqref{eq:Pt}. The dual argument has potential
order $2s-t$, while the lifting uses differences of the gradient
when $t>1$ and a separate first-derivative estimate when $t=1$.
Combining these estimates gives Theorems~\ref{t1.1} and
\ref{thm:higher}. Bounded
truncations converge in the asserted Sobolev space, so the weak
equation passes to the limit without assuming energy membership.
The explicit choice of the data range then proves
Corollary~\ref{c1.1}.

The remainder of this paper is organized as follows.
Section~\ref{s2} shows the boundary and annular-kernel estimates. 
Section~\ref{s3} establishes the
local potential estimates, the weighted zero-extension estimate,
and the Sobolev lifting.
Section~\ref{s4} proves Proposition~\ref{prop:lifting}
and uses it to prove Theorems~\ref{t1.1} and
\ref{thm:higher} and Corollary~\ref{c1.1}.
Section~\ref{s5} constructs the
counterexamples in Theorem~\ref{thm:optimal}.

Finally, we make some conventions on notation.
We use $B(x,r)=B_r(x)$ for Euclidean balls
with center $x$ and radius $r$. set $B_r=B_r(0)$.
For the fixed domain $\Omega$, write $D=\diam(\Omega)$ and
$\delta(x)=\dist(x,\partial\Omega)$ for $x\in\Omega$.
The symbols $|E|$ and $\mathcal H^{n-1}$ denote Lebesgue measure
and $(n-1)$-dimensional Hausdorff measure, respectively.
We write $x=(x',x_n)\in\R^{n-1}\times\R$.
For a measurable set $E$, $\one_E$ denotes its indicator, and
$\fint_E h\,d\mu=\mu(E)^{-1}\int_Eh\,d\mu$.
For $q\in[1,\infty]$, let $q':=q/(q-1)$.

Throughout the positive estimates, $s$ is the fractional order,
$t$ the differentiability order, $p$ the target exponent, and
$q$ the data exponent. Capital $P$ denotes an auxiliary annular
moment exponent. The potential gain is $\alpha=2s-t$.
The exponent $\varepsilon_\Omega$ belongs to the
two-sided boundary estimate in Lemma~\ref{le2.4}, whereas
$\beta_\Omega$ is the upper decay exponent in
Lemma~\ref{le2.9}. These two exponents are not identified.
Data in whole-space potentials are extended by zero.

The positive constants $C$ and $c$ may change from line to line.
Their dependencies are stated in each result. Dependence on
$\Omega$ includes its diameter and Lipschitz character, including
a localization radius. We write $A\sim B$ if $cB\le A\le CB$
with constants independent of the variables under consideration.

\section{Preliminaries}\label{s2}

\subsection{Fractional harmonic functions on Lipschitz domains}\label{s2.1}

We specify the meaning of fractional harmonicity and list some necessary
boundary estimates.

\begin{definition}[$s$-harmonic functions]\label{def:sharm}
Let $n\ge2$, $s\in(0,1)$, and $O\subset\R^n$ be a domain. A function
$v\in L^1(\R^n,w_s)$ is \emph{$s$-harmonic in $O$} if,
for any function $\phi\in C_{\rm c}^\infty(O)$,
\begin{equation*}
 \int_{\R^n}v(x)(-\Delta)^s\phi(x)\,dx=0.
\end{equation*}
Here $w_s(x)=(1+|x|)^{-n-2s}$. Thus $s$-harmonicity is precisely
the distributional equation $(-\Delta)^sv=0$ in $O$.
\end{definition}

\begin{remark}
The function $v$ is defined on $\R^n$, and the equation is imposed in $O$. The tail condition
makes the integral finite. If $v$ is $C^2$ locally in $O$, this
definition agrees with the principal-value equation from
\eqref{eq1.2}. Neither a sign condition nor zero exterior values
are part of $s$-harmonicity itself. In particular, when $O=D_r(\xi)$,
the admissible class in Definition~\ref{d1.1} additionally imposes
global nonnegativity, continuity on $B(\xi,r)$, and vanishing in
$B(\xi,r)\cap\Omega^c$. No energy-space assumption is imposed by
Definition~\ref{def:sharm}.
\end{remark}

We next recall the Lipschitz domain.

\begin{definition}[Lipschitz domains]\label{d2.1}
A bounded domain $\Omega\subset\R^n$ is called a Lipschitz domain
if there are $r_0>0$ and $L>0$ such that, for any
$\xi\in\partial\Omega$, a rigid change of coordinates taking
$\xi$ to the origin gives
\begin{equation*}
 \Omega\cap B_{r_0}
 =\{(x',x_n):x_n>\varphi_\xi(x')\}\cap B_{r_0},
\end{equation*}
where $\varphi_\xi(0)=0$,
$|\varphi_\xi(x')-\varphi_\xi(y')|\le L|x'-y'|$, and
the constants $r_0$ and $L$ are independent of $\xi$.
We may decrease $r_0$ by a factor depending only on the Lipschitz
constant $L$ of $\boz$ when applying boundary estimates.
\end{definition}

\begin{remark}\label{r2.2}
There is $c_{\rm cs}\in(0,1)$, depending only on $n$ and $\boz$, such that, for any $\xi\in\partial\Omega$
and $r\in(0,\,r_0)$, there are points $A_{\xi,r}$ and $A^{\rm ext}_{\xi,r}$ that
satisfy $B(A_{\xi,r},c_{\rm cs}r)\subset\Omega\cap B(\xi,r)$ and
$B(A^{\rm ext}_{\xi,r},c_{\rm cs}r)\subset\Omega^c\cap B(\xi,r)$.
In this case, $A_{\xi,r}$ and $A^{\rm ext}_{\xi,r}$ are called the interior and exterior corkscrew points, respectively.
The corresponding balls are called corkscrew balls. The notation
$A_{\xi,r}$ will always refer to an interior point.
\end{remark}

The Green function satisfies the following point-wise estimate (see, for instance, \cite{B99,BKN2002}).

\begin{lemma}\label{le2.5}
Let $n\ge2$, $s\in(0,1)$, and $\Omega\subset\R^n$ be a bounded
Lipschitz domain. Then, for any $x,y\in\Omega$ with $x\ne y$,
\begin{equation*}
 0\le G_s(x,y)\le C|x-y|^{2s-n},
\end{equation*}
where $C$ is a positive constant depending only on $n$, $s$, and $\Omega$.
\end{lemma}

We then recall the boundary decay estimate from
\cite{B99} (see also \cite[Theorem~2.2]{CFR26}).

\begin{lemma}[Boundary H\"older decay]\label{le2.4}
Let $n\ge2$, $s\in(0,1)$, and $\Omega\subset\R^n$ be a bounded
Lipschitz domain. Then, for any $\xi\in\partial\Omega$ and $r\in(0,\,r_0/2)$,
there are $\varepsilon_\Omega\in(0,s)$ and
$C\in(1,\,\fz)$, depending only on $n,s$ and $\Omega$, such
that every function $u$ admissible on $B(\xi,2r)$ satisfies that, for any $x\in D_r(\xi)$,
\[
 C^{-1}u(A_{\xi,r})
 \left[\frac{\delta(x)}r\right]^{2s-\varepsilon_\Omega}
 \le u(x)\le
 Cu(A_{\xi,r})\left[\frac{\delta(x)}r\right]^{\varepsilon_\Omega}.
\]
Here $D_r(\xi):=B(\xi,r)\cap\Omega$.
\end{lemma}

The next lemma is the localized boundary Harnack principle from
\cite[Theorem~2.3]{CFR26}. We normalize at the corkscrew point
$A_{\xi,r}$ from Remark~\ref{r2.2}. The epigraph formulation in
that reference gives the following version in a boundary chart of
$\Omega$, after translation and rescaling. Changing between
interior reference points at comparable scales only changes the
constants by the interior Harnack inequality.

\begin{lemma}[Boundary Harnack principle]\label{thm:bhp}
Let $n\geq2$, $s\in(0,1)$, and $\Omega$ be a bounded Lipschitz
domain.
Assume that $\xi\in\partial\Omega$ and $0<r<r_0/2$, and $u,v$
are admissible on $B(\xi,2r)$ in the sense of
Definition~\ref{d1.1}, with $u(A_{\xi,r})>0$ and
$v(A_{\xi,r})>0$. For any $x\in D_r(\xi)$, set
\[
 H_{u,v}(x):=\frac{u(x)v(A_{\xi,r})}{v(x)u(A_{\xi,r})}.
\]
Then, there exist $C\in(1,\,\fz)$ and $\vartheta\in(0,1)$, depending only
on $n,s$ and $\Omega$, such that
 $H_{u,v}$ extends to a H\"older continuous function on
$\overline{D_r(\xi)}$, and, for any $x,\,y\in D_r(\xi)$,
\begin{equation}
 C^{-1}\leq H_{u,v}(x)\leq C\qquad\text{and}\qquad
 |H_{u,v}(x)-H_{u,v}(y)|
 \leq C\left(\frac{|x-y|}{r}\right)^\vartheta.
\end{equation}
In particular, for any $x\in D_r(\xi)$,
\begin{equation}\label{eq:bhpcomp}
 C^{-1}\frac{u(A_{\xi,r})}{v(A_{\xi,r})}
 \leq\frac{u(x)}{v(x)}
 \leq C\frac{u(A_{\xi,r})}{v(A_{\xi,r})}.
\end{equation}
\end{lemma}

We also need the Green function square estimate on interior
distance level sets proved in \cite[Theorem~3.3]{CFR26}.
For $\xi\in\partial\Omega$ and $r,h>0$, denote
\[
 S_h(\xi,r):=B(\xi,r)\cap\Omega\cap\{x:\delta(x)=h\}.
\]
All averages on $S_h(\xi,r)$ below are taken with respect to
$\mathcal H^{n-1}$.

\begin{lemma}[Square estimate on parallel surfaces]\label{thm:square}
Let $n\geq2$, $s\in(0,1)$, $\Omega\subset\R^n$ be a bounded
Lipschitz domain, and $x_*\in\Omega$. Then, there exist constants
$r_1>0$ and $c_0\in(0,1)$, depending only on $\Omega$, such that, for any $\xi\in\partial\Omega$
and $r\in(0,\,r_1)$ satisfying $|x_* -\xi|>2r$,
\begin{equation}\label{eq:square}
 \left\{\fint_{S_h(\xi,r)}
       \left[\frac{G_s(x,x_*)}{h^s}\right]^2
       \,d\mathcal H^{n-1}(x)\right\}^{1/2}
 \sim\frac{G_s(A_{\xi,r},x_*)}{r^s}
\end{equation}
holds for any $h\in(0,\,c_0r)$. Here
the comparison constants depend only on $n,\,s$, and $\Omega$ and are independent of $\xi,\,r,\,h$, and
the pole $x_*$ subject to the stated separation condition.
\end{lemma}

We take the fixed localization radius $r_0$ small enough that the above
results apply in all subsequent boundary patches, in particular,
$r_0\leq r_1$.

\begin{lemma}[Upper boundary decay]\label{le2.9}
Let $n\ge2$, $s\in(0,1)$, and $\Omega\subset\R^n$ be a bounded
Lipschitz domain. Then, for any $\xi\in\partial\Omega$ and $r\in(0,\,r_0/8)$,
there are positive constants $C$ and $\beta_\Omega\in(0,s)$,
depending only on $n,s$ and $\Omega$, such that,
every function $v$ admissible on $B(\xi,8r)$ satisfies that, for any $x\in D_r(\xi)$,
\begin{equation}\label{eq2.4}
 v(x)\le Cv(A_{\xi,r})\left[\frac{\delta(x)}r\right]^{\beta_\Omega}.
\end{equation}
If $n=2$ and $s\in(1/2,\,1)$, one may choose
$\beta_\Omega>s-1/2$.
\end{lemma}

\begin{proof}
Except in the planar case with $s\in(1/2,\,1)$, we take $\beta_\Omega$ to be the exponent $\epz_\boz$ from Lemma~\ref{le2.4}. In the planar case, every boundary point
has a uniformly sized exterior sector. The complementary sector
$\mathcal C_\Theta$ can be chosen with half-opening $\Theta\in(\pi/2,\pi)$.
By strict monotonicity and the aperture limit (see, for instance,
\cite[p.~6 and Theorem~3.2]{M06}), we find that
 the $s$-harmonic function has the homogeneous of degree 
\begin{equation}\label{eq2.5}
\theta_\Omega\in\lf(s-\frac12,\,s\r).
\end{equation}

Let $\xi\in\partial\Omega$ and take $R$ such that
$\Omega\cap B(\xi,8R)\subset\xi+\mathcal C_\Theta$. Choose
$A\in\Omega\cap B(\xi,R)$ satisfying that $B(A,cR)\subset\Omega\cap B(\xi,2R)$,
where $c$ is a positive constant depending only on the Lipschitz character. From Lemma \ref{le2.4} and Harnack principle, we deduce that
\begin{equation}\label{pl:eq:carleson}
 \sup_{x\in B(\xi,4R)}v(x)\le C v(A),
\end{equation}
where $v=0$ on the local exterior is understood. By the Poisson formula in an interior ball of radius comparable to $R$ centered
at $A$ (see, for instance, \cite[Chapter 1]{bbkrs09}), we conclude that
\begin{equation}\label{pl:eq:tail}
 R^{2s}\int_{|y-\xi|>4R}
                \frac{v(y)}{|y-\xi|^{2+2s}}\dd y\le C v(A).
\end{equation}

Let $\mathcal D:=(\xi+\mathcal C_\Theta)\cap B(\xi,2R)$, $G_{s,\mathcal D}$ be Green function on $\mathcal D$, and $\widetilde v$ its
harmonic extension of the data $v\one_{B(\xi,2R)^c}$, with zero
values on the rest of $\mathcal D^c$. Comparison yields
$v\le\widetilde v$ in $\Omega\cap B(\xi,2R)$.
Assume that $A_\Theta$ is a fixed sector reference point at distance comparable
to $R$ from both the vertex and the boundary of $\mathcal D$. From the Poisson kernel representation in $\mathcal D$ (see, for instance, \cite[Chapter 1]{bbkrs09}) and the facts that 
 $G_{s,\mathcal D}\le G_{s,B(\xi,2R)}$ and
$\int_{\mathcal D} G_{s,\mathcal D}(A_\Theta,z)\dd z\le CR^{2s}$, it follows
that, for any $y$ satisfying $|y-\xi|>4R$, the Poisson
kernel $P_{\mathcal D}$ of $\mathcal D$ satisfies
\[
 P_{\mathcal D}(A_\Theta,y)= C\int_{\mathcal D} G_{s,\mathcal D}(A_\Theta,z)|z-y|^{-2-2s}\dd z
 \le C R^{2s}|y-\xi|^{-2-2s}.
\]
By this, \eqref{pl:eq:tail}, and the Poisson kernel representation, we conclude that
\[
 \widetilde v(A_\Theta)\le C v(A).
\]
The positive $s$-harmonic function $U_\Theta$ on the sector has the homogeneous of degree $\theta_\Omega$.
Boundary Harnack principle from Lemma \ref{thm:bhp}, applied to $\widetilde v$ and
$U_\Theta(\cdot-\xi)$ in the sector, yields
\begin{equation}\label{eq:seccompa}
 v(x)\le\widetilde v(x)
 \le C v(A)R^{-\theta_\Omega}U_\Theta(x-\xi)
 \le C v(A)(|x-\xi|/R)^{\theta_\Omega}.
\end{equation}
To obtain the estimate on $D_r(\xi)$, let $\zeta$ be a nearest
boundary point to $x\in D_r(\xi)$ and apply the preceding argument
at $\zeta$ with $R=r/4$. Since $|\zeta-\xi|\le2r$, all the
required patches lie in $B(\xi,8r)$. From \eqref{pl:eq:carleson}, it
follows that $v(A_{\zeta,R})\le Cv(A_{\xi,r})$. If
$\delta(x)\in(0,\,c_\Omega r)$, where $c_\Omega\in(0,\,1)$ sufficiently small, then,
by \eqref{eq:seccompa} and taking $\bz_\boz=\theta_\Omega$,  we find that \eqref{eq2.4}
holds. If $\delta(x)\ge c_\Omega r$, then the same
estimate follows directly from \eqref{pl:eq:carleson}. This completes
the proof of Lemma \ref{le2.9}.
\end{proof}

\begin{remark}
Compared to Lemma \ref{le2.4}, only the upper estimate is changed
in Lemma \ref{le2.9} and the lower estimate in
Lemma~\ref{le2.4} retains its original exponent. In dimensions
$n\ge3$, the corresponding cone exponent tends to zero as the
exterior cone narrows (see, for instance, \cite[Theorem~3.2]{M06}). Thus a
fixed exterior cone still gives a domain-dependent gain, but does
not raise the universal baseline in \eqref{eq:Pt}.
\end{remark}

\subsection{Annular kernel estimates and the dual argument}

The next lemma converts the boundary reverse H\"older estimate into an annular estimate.
Define the positive kernel, for any $x,\,y\in\Omega$,
\[
 \mathcal K_s(x,y):=\frac{G_s(x,y)}{\delta(x)^s}
\]
and extend it by zero when either variable is outside $\Omega$.

\begin{lemma}[Uniform annular estimate]\label{le2.6}
Let $n\ge2$, $s\in(0,1)$, and $P\in(2,\,\infty)$. Assume that \eqref{eq1.4} holds with exponent $P$. Then,
for any $y\in\R^n$ and $r\in(0,\,\infty)$,
\begin{equation}\label{e2.3}
 \left(\int_{\{r<|x-y|\leq2r\}}[\mathcal K_s(x,y)]^P\,dx\right)^{1/P}
 \leq C r^{s-n+n/P}
\end{equation}
where $C$ depends only on $n,s,P,\Omega$ and the constant in \eqref{eq1.4}.
\end{lemma}

\begin{proof}
For $y\notin\Omega$ the assertion is trivial. We only need
to consider the case $y\in\Omega$. Let $D:=\diam(\Omega)$.
If $y\in\Omega$ and $r\geq D$, the annulus has empty intersection with $\Omega$.
Choose a fixed $\theta\in(0,\,1/100)$ so small that $8\theta D\in(0,\,r_0)$,
where $r_0$ is the localization radius in Definition~\ref{d2.1}.

For any $r\in(0,\,D)$, let $\rho:=\theta r$ and
$\mathcal A(y,r):=\Omega\cap\{r<|x-y|\leq2r\}$.
From Lemma~\ref{le2.5}, we deduce that, for any $x\in\mathcal A(y,r)$ with $\delta(x)\ge\rho/8$,
\begin{equation*}
 \mathcal K_s(x,y)\le C  r^{2s-n}\rho^{-s}\le C  r^{s-n},
\end{equation*}
Consequently,
\begin{equation}\label{e2.4}
 \|\mathcal K_s(\cdot,y)\|_{L^P(\mathcal A(y,r)\cap\{\delta(\cdot)\geq\rho/8\})}
 \le C  r^{s-n}|\mathcal A(y,r)|^{1/P}
 \le C  r^{s-n+n/P}.
\end{equation}

To cover the remaining part $\{x:\delta (x)<\rho/8\}\cap\mathcal A(y,r)$, consider
\[
 \Gamma=\{\xi\in\partial\Omega:r/2\leq|\xi-y|\leq3r\}.
\]
Take a maximal $\rho/4$-separated family $\{\xi_j\}_{j=1}^N$ in
$\Gamma$. Disjointness of the balls $B(\xi_j,\rho/8)$ and volume
comparison in $B(y,4r)$ imply
$N\le C(r/\rho)^n=C\theta^{-n}$.
If $x\in \mathcal A(y,r)$ and $\delta(x)\in(0,\,\rho/8)$, then a nearest boundary point
$\xi$ belongs to $\Gamma$ and is within $\rho/4$ of some $\xi_j$.
By this, we find that $x\in B(\xi_j,\rho)$.
Applying the pointwise estimate of the Green function in Lemma \ref{le2.5} and the boundary H\"older theorem from Lemma \ref{le2.4} in $B(\xi_j,2\rho)$,
we obtain that, for any $x\in B(\xi_j,\rho)\cap\Omega$,
\begin{equation*}
 \mathcal K_s(x,y)\le C  r^{2s-n}\rho^{-s}
               (\delta(x)/\rho)^{\varepsilon_\Omega-s}.
\end{equation*}
From this, the coarea formula, and the fact that $s-\varepsilon_\Omega<1$, it follows that
\begin{align}\label{eq2.6}
\int_{B(\xi_j,\rho)\cap\Omega}\mathcal K_s(x,y)\,dx
&\le C  \int_{B(\xi_j,\rho)\cap\Omega} r^{2s-n}\rho^{-s}
               (\delta(x)/\rho)^{\varepsilon_\Omega-s}\,dx\\\nonumber
&\le C  r^{2s-n}\rho^{n-1-s}\int_{0}^{\rho}(h/\rho)^{\varepsilon_\Omega-s}\,dh
\le C  r^{2s-n}\rho^{n-s}.
\end{align}
Since $|\xi_j-y|\geq r/2>8\rho$ and $8\rho<r_0$,
$G_s(\cdot,y)$ is admissible on $B(\xi_j,8\rho)$.
By \eqref{eq1.4} with radius $8\rho$ and \eqref{eq2.6}, we conclude that
\begin{align*}
\int_{B(\xi_j,\rho)\cap\Omega}[\mathcal K_s(x,y)]^P\,dx
&\le C  \rho^{n-nP}\left[\int_{B(\xi_j,\rho)\cap\Omega}\mathcal K_s(x,y)\,dx\right]^P\\\nonumber
&\le C  \rho^{n-nP}r^{P(2s-n)}\rho^{P(n-s)}
\le C  r^{P(2s-n)}\rho^{n-Ps}.
\end{align*}

Summing these bounds over the $N$ balls, and recalling
$\rho=\theta r$, yields
\[
 \int_{\mathcal A(y,r)\cap\{\delta(\cdot)<\rho/8\}}[\mathcal K_s(x,y)]^P\,dx
 \le C  r^{(s-n)P+n}.
\]
This, together with \eqref{e2.4}, implies that 
\eqref{e2.3} holds. This completes the proof of Lemma \ref{le2.6}.
\end{proof}

We next give an abstract consequence of an annular moment bound.
It will be applied to the Green quotients of orders $s$ and $t$.

\begin{lemma}[A bound from annular moments]\label{le2.7}
Let $n\ge2$, $\alpha\in(0,\,n)$, and $P\in(1,\infty)$. Suppose that a nonnegative
measurable kernel $K(\cdot,\,\cdot)$ on $\R^n\times\R^n$ satisfies that, for any
$y\in\R^n$ and $r\in(0,\,\fz)$,
\begin{equation}\label{eq:abstractannular}
 \|K(\cdot,y)\|_{L^P(\{r<|x-y|\le2r\})}
 \le C_Kr^{\alpha-n+n/P}.
\end{equation}
Here $C_K$ is a positive constant independent of $r$.
If $q\in(1,\,n/\alpha)$, $p\in(1,\,P)$, and $1/p=1/q-\alpha/n$, then, for any $f\in L^q(\R^n)$,
\begin{equation}\label{eq2.8}
 \left\|\int_{\R^n}K(\cdot,y)f(y)\,dy\right\|_{L^p(\R^n)}
 \le C\|f\|_{L^q(\R^n)},
\end{equation}
where the positive constant $C$ depends only on $n,\,\alpha,\,P,\,p,\,q$, and $C_K$.
\end{lemma}
\begin{proof}
Let $r_*=P'$. Then
$
 1<r_*<p'<\frac n\alpha
$
and
$
\frac1{q'}=\frac1{p'}-\frac\alpha n
$.
For a function $g\in L^{p'}(\R^n)$, define
\[
 T^*g(y):=\int_{\R^n}K(x,y)g(x)\,dx
\]
and
\[
 \mathcal M_{r_*}g:=(\mathcal M(|g|^{r_*}))^{1/r_*},
\]
where
\[
 \mathcal Mg(y):=\sup_{R>0}\fint_{B(y,R)}|g(x)|\,dx
\]
is the centered Hardy--Littlewood maximal operator. 
From H\"older's inequality and \eqref{eq:abstractannular}, we deduce that
\begin{equation*}
 T^*g(y)\le C\sum_{j\in\Z}2^{j\alpha}
 \left(\fint_{B(y,2^{j+1})}|g(x)|^{r_*}\,dx\right)^{1/r_*}.
\end{equation*}
For $R>0$, the sum over $2^j\le R$ is at most
$CR^\alpha\mathcal M_{r_*}g(y)$. Since $r_*<p'$, the remaining
averages satisfy
\[
 \left(\fint_{B(y,2^{j+1})}|g(x)|^{r_*}\,dx\right)^{1/r_*}
 \le C2^{-jn/p'}\|g\|_{L^{p'}(\R^n)}.
\]
By this and the fact that $\alpha-n/p'<0$, we conclude that
\begin{equation*}
 T^*g(y)\le C\left(R^\alpha\mathcal M_{r_*}g(y)
          +R^{\alpha-n/p'}\|g\|_{L^{p'}(\R^n)}\right).
\end{equation*}
If $g\not\equiv0$, choose
$R=(\|g\|_{L^{p'}}/\mathcal M_{r_*}g(y))^{p'/n}$ at almost every
point. Then
\begin{equation*}
 T^*g(y)\le C\|g\|_{L^{p'}(\R^n)}^{\alpha p'/n}
             (\mathcal M_{r_*}g(y))^{1-\alpha p'/n}.
\end{equation*}
This, together with the boundedness of the Hardy--Littlewood maximal operator
(see, for instance, \cite[Theorem~2.1.6]{gra14}) and the fact that $(1-\alpha p'/n)q'=p'$, 
implies that
\begin{equation*}
 \|T^*g\|_{L^{q'}(\R^n)}\le C\|g\|_{L^{p'}(\R^n)}.
\end{equation*}
Tonelli's theorem and duality now give \eqref{eq2.8}.
This finishes the proof of Lemma~\ref{le2.7}.
\end{proof}

\section{From the boundary quotient to global Sobolev regularity}\label{s3}
Throughout this section, let $t\in[s,\,\min\{s+1/2,2s\})$,
$\alpha:=2s-t$, and $D:=\diam(\Omega)$. Constants may depend on
$n,s,t,p,q$ and the fixed Lipschitz geometry, but not on the
Whitney scale or on the functions under consideration.

\subsection{Local potential and harmonic estimates}
For a cube $Q$ at scale $\ell$ and a function
$v$ on $Q$, introduce the scaled quantity
\[
 N_{t,p,\ell}(v;Q)=
 \begin{cases}
 \ell^{-t}\|v\|_{L^p(Q)}+[v]_{W^{t,p}(Q)},&t\in(0,\,1),\\
 \ell^{-1}\|v\|_{L^p(Q)}+\|\nabla v\|_{L^p(Q)},&t=1,\\
 \ell^{-t}\|v\|_{L^p(Q)}+\ell^{1-t}\|\nabla v\|_{L^p(Q)}
       +[\nabla v]_{W^{t-1,p}(Q)},&t\in(1,\,2).
 \end{cases}
\]
All terms have the scaling of a derivative of order $t$.
Let $\Delta_hv(x):=v(x+h)-v(x)$ be the first differences.

\begin{lemma}\label{le3.1}
Let $1<q<p<\infty$ and $0<\gamma=n(1/q-1/p)\le2s-t$.
Then, for any concentric cubes $Q^1$ and $Q^2$ satisfying that $Q^1\Subset Q^2$ with fixed size ratios and
side lengths comparable to $\ell$, and $F$ supported in $Q^2$,
\begin{equation}\label{eq3.1}
 N_{t,p,\ell}(I_{2s}F;Q^1)
 +\ell^{-t}\|I_{2s}F\|_{L^p(Q^2)}
 \le C\ell^{2s-t-\gamma}\|F\|_{L^q(Q^2)},
\end{equation}
where the constant $C$ depends only on $n,s,t,p,q$ and the fixed cube ratios.
\end{lemma}
\begin{proof}
Let $K_{2s}(x)=c(n,s)|x|^{2s-n}$  be the normalized Riesz
kernel. Then $(-\Delta)^sI_{2s}F=F$.
We first consider the case $t\ne1$. 
Let $j=0$ if $t\in(0,\,1)$ and $j=1$ if $t\in(1,\,2)$,
and set $\tau:=t-j\in(0,1)$. Then each component of $\partial^jK_{2s}$
is homogeneous of degree $2s-j-n$, where $\partial^j$ denotes the
identity operator if $j=0$ and the first-order differential operator
if $j=1$.

On the critical line $\gamma=2s-t$, choose $q_0$ and $q_1$ satisfy
$1<q_0<q<q_1<p$ and sufficiently close to $q$. Define $r_i$ and $\sigma_i$
by $1+1/p=1/q_i+1/r_i$ and $\sigma_i=2s-j-n(1/q_i-1/p)$.
Then
\[
0<\sigma_0<\tau<\sigma_1<1.
\]
From scaling, local integrability at the two singularities, and decay
at infinity, it follows that
\[
 \|\Delta_h\partial^jK_{2s}\|_{L^{r_i}(\R^n)}
 \le C|h|^{\sigma_i},\qquad i=0,1.
\]
Then, by rotation invariance, we find that the positive constant $C$ is
independent of the direction of $h$.
For a decomposition $F=F_0+F_1$ with $F_0\in L^{q_0}(\mathbb{R}^n)$
and $F_1\in L^{q_1}(\mathbb{R}^n)$, Young's inequality yields
\begin{equation}\label{eq3.2}
 \|\Delta_h\partial^jI_{2s}F\|_{L^p(\mathbb{R}^n)}
 \le C\bigl(|h|^{\sigma_0}\|F_0\|_{L^{q_0}(\mathbb{R}^n)}
              +|h|^{\sigma_1}\|F_1\|_{L^{q_1}(\mathbb{R}^n)}\bigr).
\end{equation}

For the fixed pair $(L^{q_0}(\R^n),L^{q_1}(\R^n))$ and $\kappa\in(0,\,\fz)$, define
the real-interpolation functional
\[
 \mathfrak K(\kappa,f;L^{q_0},L^{q_1}):=\inf_{f=f_0+f_1}
          \lf[\|f_0\|_{L^{q_0}(\R^n)}+\kappa\|f_1\|_{L^{q_1}(\R^n)}\r].
\]
Here, in taking the infimum, $F_0$ and $F_1$ may be restricted to the fixed support
of $f$. Thus all potentials in this calculation
are locally well defined. Then
\eqref{eq3.2} yields
\begin{equation*}
 \|\Delta_h\partial^jI_{2s}F\|_{L^p(\R^n)}
 \lesssim |h|^{\sigma_0}\mathfrak K(|h|^{\sigma_1-\sigma_0},F;L^{q_0},L^{q_1}).
\end{equation*}
Let $d:=\sigma_1-\sigma_0$ and $\theta:=(\tau-\sigma_0)/d$.
By taking polar integration in $h$, we find that
\begin{align}\label{eq3.5}
 [\partial^jI_{2s}F]_{W^{\tau,p}(\R^n)}^p
 &\lesssim \int_0^\infty
       \rho^{(\sigma_0-\tau)p}\lf[\mathfrak K(\rho^d,F;L^{q_0},L^{q_1})\r]^p\,\frac{d\rho}{\rho}
       \\\nonumber
 &=\frac 1d\int_0^\infty
       \lf[t^{-\theta}\mathfrak K(t,F;L^{q_0},L^{q_1})\r]^p\,\frac{d t}{t}.
\end{align}

Recall that, for a function $F$ on $\mathbb{R}^n$, the distribution function
and the decreasing rearrangement of $F$ are defined by
\[
\mu_F(a):=|\{x\in\R^n:|F(x)|>a\}|,\quad a>0
\]
and
\[
 F^*(t):=\inf\{a>0:\mu_F(a)\leq t\},\quad t>0,
\]
respectively. Then, from this, it follows that
the function $F^*$ is nonincreasing and equimeasurable with $|F|$ and
\begin{equation}\label{eq3.6}
|\{t>0:F^*(t)>a\}|=\mu_F(a).
\end{equation}

For any $p,\,q\in(1,\,\fz)$,
the Lorentz space $L^{q,p}(\R^n)$ is defined by
$$
L^{q,p}(\R^n):=\lf\{F:\|F\|_{L^{q,p}(\R^n)}<\fz\r\},
$$
where we use the standard rearrangement quasi-norm
\begin{equation*}
 \|F\|_{L^{q,p}(\R^n)}
 :=\left(\int_0^\infty
        [t^{1/q}F^*(t)]^p\,\frac{d t}{t}\right)^{1/p}.
\end{equation*}

By \eqref{eq3.5}, the real-interpolation identity
$(L^{q_0}(\mathbb{R}^n),L^{q_1}(\mathbb{R}^n))_{\theta,p}=L^{q,p}(\mathbb{R}^n)$
(see, for instance, \cite[Theorem~5.3.1]{bl76}), and the parameter identity
\[
 \frac1q=\frac{1-\theta}{q_0}+\frac{\theta}{q_1},
\]
we conclude that
\begin{equation}\label{eq3.8}
[\partial^jI_{2s}F]_{W^{\tau,p}(\R^n)}^p\lesssim
\int_0^\infty[t^{1/q}F^*(t)]^p\,\frac{d t}{t}.
\end{equation}

Let $A(t):=t^{1/q}F^*(t)$. For any $F\in L^m(\mathbb{R}^n)$ with $m\in(1,\fz)$,
the layer-cake formula provides
\begin{equation*}
\int_\mathbb{R}^n |F(x)|^m\,dx
=m\int_0^\fz a^{m-1}\mu_F(a)\,da.
\end{equation*}
From this and \eqref{eq3.6}, we deduce that
\begin{align*}
\int_0^\fz |F^*(t)|^q\,dt
&=q\int_0^\fz a^{q-1}|\{t>0:F^*(t)>a\}|\,da\\\nonumber
&=q\int_0^\fz a^{q-1}\mu_F(a)\,da=\|F\|^q_{L^q(\mathbb{R}^n)}.
\end{align*}
By this and the monotonicity of $f^*$, we find
that
\begin{equation*}
 A(t)^q=tF^*(t)^q\leq\int_0^tF^*(u)^qdu
 \leq\|F\|_{L^q(\R^n)}^q
\end{equation*}
and
\begin{equation*}
 \int_0^\infty A(t)^q\,\frac{\dd t}{t}=\|F\|_{L^q(\R^n)}^q,
\end{equation*}
which, together with \eqref{eq3.8}, further imply that
\begin{equation*}
[\partial^jI_{2s}F]_{W^{\tau,p}(\R^n)}^p
\ls\int_0^\infty A(t)^{p-q}A(t)^q\,\frac{\dd t}{t}
 \leq\|F\|_{L^q(\R^n)}^{p-q}\|F\|_{L^q(\R^n)}^q=\|F\|_{L^q(\R^n)}^{p}.
\end{equation*}

For the subcritical case $\gamma=n(1/q-1/p)\in(0,\,2s-t)$, we first
normalize the cube scale to $\ell=1$. Choose a fixed smooth cutoff
$\chi$ which satisfies $\chi=1$ on domain $\{x-y:x\in Q^2\ \text{and}\ y\in Q^2\}$ and is supported in $B_{R_0}$
for a fixed $R_0>0$, and set $J=\chi K_{2s}$. Define $m$ by
$1+1/p=1/q+1/m$ and write
$\Delta_h\partial^jJ(x):=\partial^jJ(x+h)-\partial^jJ(x)$. For any
$\sigma\in(\tau,\,\min\{1,2s-j-\gamma\})$, by splitting the integral,
we conclude that
\[
 \|\partial^jJ\|_{L^m(\mathbb{R}^n)}\le C
 \]
 and
 \[
 \|\Delta_h\partial^jJ\|_{L^m(\mathbb{R}^n)}\le C\min\{|h|^\sigma,1\}.
\]
Indeed, when $0<|h|<1/4$, the kernel bound
$|\partial^jK_{2s}(x)|\le C|x|^{2s-j-n}$ bounds the near part by
$C|h|^{2s-j-\gamma}$. The mean value theorem and the derivative
bound for $\partial^jK_{2s}$ bound the far part by
\[
 C|h|\left(1+
 \left[\int_{2|h|}^{R_0+1}
 r^{m(2s-j-n-1)+n-1}\dd r\right]^{1/m}\right)\le C|h|^\sigma.
\]
From Young's convolution inequality, it follows that
$\|\Delta_h\partial^j(J*F)\|_{L^p(\mathbb{R}^n)}\le
\|\Delta_h\partial^jJ\|_{L^m(\mathbb{R}^n)}\|F\|_{L^q(\mathbb{R}^n)}$.
By the fact that $\tau\in(0,\,\sigma)$ and
integrating its $p$-th power against $|h|^{-n-\tau p}\dd h$,
we find that the subcritical estimate holds.

For the remaining integer case $t=1$, the parameter range implies
$s\in(1/2,\,1)$.
On the critical line $\gamma=2s-1$, the pointwise bound for the
gradient of the Riesz kernel and the Hardy--Littlewood--Sobolev
inequality (see, for instance, \cite[Chapter~V]{ste70}) yield
\[
 |\nabla I_{2s}F|\le C I_{2s-1}|F|,
 \qquad \|\nabla I_{2s}F\|_{L^p}\le C\|F\|_{L^q}.
\]
When $\gamma\in(0,\,2s-1)$, the truncated
gradient kernel satisfies $\|\nabla J\|_{L^m(\R^n)}\le C$, so Young's inequality applies.

For the lower-order terms in either case, work at unit scale
and use the same cutoff $J=\chi K_{2s}$, where
$1+1/p=1/q+1/m$. For any integer $k\in[0,\,t)$,
one has $2s-k-\gamma>0$, then by applying Young's convolution
inequality to $\partial^k(J*F)=(\partial^kJ)*F$, we find that
$$
 \|\partial^kJ\|_{L^m(\R^n)}\le C
$$
and
$$
\|\partial^k(J*F)\|_{L^p(\R^n)}\le C\|F\|_{L^q(\R^n)}.
$$
This, together with the fact that $\chi=1$ on domain $\{x-y:x\in Q^2\ \text{and}\ y\in Q^2\}$,
implies the lower-order estimates hold.

Finally, let $x_Q$ be the center of $Q$ and rescale by
$x:=x_Q+\ell z$. The left-hand side of \eqref{eq3.1}
scales as $\ell^{2s-t+n/p}$, while the $L^q$ data norm scales as
$\ell^{n/q}$. Since $\gamma=n(1/q-1/p)$, their ratio gives the factor
$\ell^{2s-t-\gamma}$.
For general data, approximate $F$ by smooth compactly supported
functions. Local convergence of the potentials identifies the
distributional derivatives. Fatou's lemma for fractional seminorms
and weak lower semicontinuity when $t=1$ then give
\eqref{eq3.1}. This finishes the proof of Lemma \ref{le3.1}.
\end{proof}

\begin{lemma}\label{le3.2}
Let $Q^1\Subset Q^2$ be concentric cubes of side lengths
comparable to $\ell$, with fixed size ratios. Suppose that
$v\ge0$ on $\R^n$, $v\in L^1(\R^n,w_s)$, and $v$ is
$s$-harmonic on $Q^2$. Then, for any $t\in(0,2)$ and $p\in(1,\,\fz)$,
\begin{equation}\label{e3.5}
 N_{t,p,\ell}(v;Q^1)\le C\ell^{-t}\|v\|_{L^p(Q^2)}.
\end{equation}
Here the positive $C$ depends only on $n$, $s$, $t$, $p$, and the fixed cube ratios.
\end{lemma}

\begin{proof}
First suppose that $v$ is smooth in $Q^2$. Let
$B=B(z,R)\Subset Q^2$. By the Poisson representation for a ball
(see, for instance, \cite[Theorem~2.10]{Bu16}), for any $x\in B$,
\[
 v(x)=\int_{\R^n\setminus\overline B}P_B(x,y)v(y)\,dy,
\]
where $P_B$ denotes the fractional Poisson kernel. For any $x\in B$
and $y\notin\overline B$, the fractional Poisson kernel is given by
\[
 P_B(x,y)=\frac{\Gamma(n/2)\sin(\pi s)}{\pi^{n/2+1}}
 \frac{(R^2-|x-z|^2)^s}{(|y-z|^2-R^2)^s|x-y|^n}.
\]
For any $x\in B(z,R/2)$, by differentiating the logarithm of the
displayed kernel, we find that $|\nabla_x\log P_B|\le C/R$. A second
differentiation yields
\begin{align}\label{eq3.7}
 \partial_x^2\log P_B={}&-\frac{2sI}{R^2-|x-z|^2}
 -\frac{4s(x-z)\otimes(x-z)}{(R^2-|x-z|^2)^2}\\\notag
 &-\frac{nI}{|x-y|^2}
 +\frac{2n(x-y)\otimes(x-y)}{|x-y|^4}.
\end{align}
Here $I$ is the $n\times n$ identity matrix and $a\otimes b$
denotes the tensor product of vectors $a,b\in\R^n$. Since
$|x-y|\ge R/2$ on the half-ball, \eqref{eq3.7} implies
$|\partial_x^kP_B|\le C R^{-k}P_B$ for $k=1,2$, where $C$ depends only on $n$ and $s$,
which justifies differentiation under the integral. From this and the nonnegativity of $v$, 
we deduce that, for any
$x,x'\in B(z,R/2)$,
$
 |\partial^kv(x)|\le CR^{-k}v(x)$ and
 $
 v(x)\le C v(x').
$
By integrating the second inequality over a ball contained in
$Q^2$ and then using a finite covering of $Q^1$ by half-balls
whose radii are comparable to $\ell$, we conclude that
\begin{equation}\label{e3.8}
 \|\partial^kv\|_{L^\infty(Q^1)}
 \le C\ell^{-k-n/p}\|v\|_{L^p(Q^2)},\qquad k=1,2.
\end{equation}
If $t\in(0,1)$, then the mean value theorem and the bound for $\nabla v$
control $|v(x)-v(y)|$ by $C|x-y|\ell^{-1-n/p}\|v\|_{L^p(Q^2)}$.
If $t\in(1,2)$, the same argument applied to $\nabla v$ uses the bound
for $\partial^2v$. In the two cases, radial integration in the fractional
seminorm produces the integrable powers $p(1-t)-1$ and
$p(2-t)-1$, respectively. For $t=1$, only the $L^p$ gradient
bound is required.
This proves \eqref{e3.5} for smooth $v$.

For a distributional $s$-harmonic function $v$, let
$v_\varepsilon:=v*\eta_\varepsilon$, where $\eta_\varepsilon$ is a
nonnegative mollifier of radius smaller than the fixed separation
between $Q^1$ and $\partial Q^2$. These mollifications remain
nonnegative and have integrable fractional tails. They are
$s$-harmonic on the interior balls used above. By Young's
inequality, their $L^p$ norms on intermediate cubes are bounded
by $\|v\|_{L^p(Q^2)}$. The first two derivative estimates \eqref{e3.8} yield
weak compactness in $W^{2,p}$ on nested interior cubes. The
distributional limit identifies the derivatives of $v$, and lower
semicontinuity gives \eqref{e3.5}, with constants
independent of $\varepsilon$. This finishes the proof of Lemma \ref{le3.2}.
\end{proof}

\subsection{Zero extension and Sobolev lifting}
We first establish the weighted estimate for zero extensions that
will give the converse implication in Proposition~\ref{prop:lifting}.

\begin{lemma}\label{le3.3}
Let $n\ge2$, $t\in(0,\,2)$, $p\in(1,\infty)$, and $\Omega\subset\rn$ be a bounded Lipschitz domain. If $w\in W^{t,p}(\R^n)$ and $w=0$ almost everywhere
on $\Omega^c$, then
\begin{equation}\label{e3.7}
 \|w/\delta^t\|_{L^p(\Omega)}\le C\|w\|_{W^{t,p}(\R^n)},
\end{equation}
where the positive constant $C$ depends only on $n$, $t$,
$p$, and $\Omega$.
\end{lemma}

\begin{proof}
For $t\ne1$ and $x,h\in\R^n$, define the second difference by
$\Delta_h^2w(x):=w(x+2h)-2w(x+h)+w(x)$. If $t\in(0,1)$,
from the translation invariance and the triangle inequality, we deduce that
\begin{equation}\label{equ3.8}
 \int_{\R^n}\frac{\|\Delta_h^2w\|_{L^p(\mathbb{R}^n)}^p}{|h|^{n+tp}}\,dh
 \le 2^p[w]_{W^{t,p}(\R^n)}^p.
\end{equation}
If $t\in(1,2)$, by the fundamental theorem of calculus, we conclude that
\[
 \Delta_h^2w(x)=\int_0^1
 [\nabla w(x+h+\tau h)-\nabla w(x+\tau h)]\cdot h\,d\tau.
\]
Consequently,
\begin{equation}\label{equ3.9}
 \int_{\R^n}\frac{\|\Delta_h^2w\|_{L^p(\mathbb{R}^n)}^p}{|h|^{n+tp}}\,dh
 \le [\nabla w]_{W^{t-1,p}(\R^n)}^p.
\end{equation}

We next establish the geometric lower bound. In a boundary chart,
write $\Omega=\{(x',x_n):x_n>\varphi(x')\}$ locally, where
$\varphi:\R^{n-1}\to\R$ is an $L$-Lipschitz function. For
$x=(x',x_n)$ in a smaller chart, set
$a:=x_n-\varphi(x')>0$. Then
\[
 \delta(x)\le a\le\sqrt{1+L^2}\,\delta(x).
\]
For any $x=(x',x_n)$ in a smaller chart with sufficiently small $a$, define
\[
 E_x:=\{h:|h'|<a,\;-(L+4)a<h_n<-(L+3)a\}.
\]
Then, for any $h\in E_x$,  $x+h,x+2h\in\Omega^c$.
From this and the facts that $w(x)=0$ on $\boz^c$, $|E_x|\sim a^n$, and $|h|\sim a$ on $E_x$,
it follows that
\begin{equation}\label{e3.12}
 \int_{E_x}\frac{|\Delta_h^2w(x)|^p}{|h|^{n+tp}}\,dh
 \ge c\frac{|w(x)|^p}{\delta(x)^{tp}}.
\end{equation}
Here the positive constant $c$ depends only on $n$, $t$, $p$, and $\Omega$.
Integrating \eqref{e3.12} in $x$
over a finite cover of smaller boundary charts and using
\eqref{equ3.8} or \eqref{equ3.9}
controls the weighted norm in a boundary strip. On the complement
of this strip, $\delta$ has a positive lower bound, so the weighted
norm is at most $C\|w\|_{L^p(\R^n)}$. This proves
\eqref{e3.7} when $t\ne1$.

For $t=1$, localize with a smooth partition of unity in the same
charts. On almost every vertical line the zero extension of $w$ is a
$W^{1,p}$ function. By integration by parts and H\"older's inequality,
we find that, for any compactly supported $g\in W^{1,p}(0,\infty)$,
the one-dimensional Hardy inequality
\begin{equation}\label{e3.13}
 \int_0^\infty\left|\frac{g(a)}a\right|^pda
 \le\left(\frac p{p-1}\right)^p
       \int_0^\infty|g'(a)|^pda,
\end{equation}
holds. Since
$a\sim\delta(x)$ in the smaller chart, from applying \eqref{e3.13} on
almost every vertical line to
$g(a)=(\chi w)(x',\varphi(x')+a)$, where $\chi$ is a member of
the partition of unity, and integrating in $x'$, it follows that the desired weighted estimate
holds for $g$. By this, the fact that terms involved derivatives of $\chi$
contribute at most $C\|w\|_{L^p(\mathbb{R}^n)}$, and summing over the finitely many charts,
we conclude that \eqref{e3.7} holds for $t=1$. This finishes the proof
of Lemma \ref{le3.3}.
\end{proof}

We now combine the local estimates with Whitney localization to
prove Proposition~\ref{prop:lifting}.
Let $\{Q\}$ be a Whitney decomposition of $\Omega$, and denote
the side length of $Q$ by $\ell(Q)$. For each $Q$, choose concentric
dilations $Q^1$ and $Q^2$ with fixed dilation factors such that
\[
Q\Subset Q^1\Subset Q^2\Subset\Omega,
\]
the family $\{Q^2\}$ has bounded overlap, and
$\delta(x)\sim\ell(Q)$ for all $x\in Q^2$.
Throughout the proof, we set $\alpha:=2s-t$ and
$\gamma:=n(1/q-1/p)$, as in Proposition~\ref{prop:lifting}.

\begin{proof}[Proof of Proposition~\ref{prop:lifting}]
First let $f\ge0$ be bounded and write $u=\G f$. For every
Whitney cube $Q$, decompose the datum and the potential by setting
\[
 F_Q=f\one_{Q^2},\quad v_Q=\G F_Q,\quad
 b_Q=u-v_Q,\quad \Phi_Q=I_{2s}F_Q,\quad z_Q=\Phi_Q-v_Q.
\]
Then $b_Q$ and $z_Q$ are $s$-harmonic on $Q^2$.
From the positivity and point-wise estimate of Green function in 
Lemma \ref{le2.5}, we deduce that
$0\le v_Q\le u$, $0\le v_Q\le\Phi_Q$, and hence
$b_Q,z_Q\ge0$ globally. 
Since $u$ and $v_Q$ are
integrable and supported in $\overline\Omega$, while $\Phi_Q$
has Riesz decay, the fractional tails of $b_Q$ and $z_Q$ are integrable.
By this and applying Lemmas~\ref{le3.1} and
\ref{le3.2} to the identity
$u=\Phi_Q-z_Q+b_Q$, and using $z_Q\le\Phi_Q$ and $b_Q\le u$, we find that
\begin{equation*}
 N_{t,p,\ell(Q)}(u;Q^1)
 \le C\left(\ell(Q)^{\alpha-\gamma}\|f\|_{L^q(Q^2)}
                      +\ell(Q)^{-t}\|u\|_{L^p(Q^2)}\right).
\end{equation*}
This, together with the facts that the bounded overlap of the dilated Whitney cubes, $p>q$, $\az\ge\gz$,
and $\ell(Q)\le \diam(\boz)$, implies that
\begin{align}
 \sum_Q N_{t,p,\ell(Q)}(u;Q^1)^p
 &\le C \|f\|_{L^q(\Omega)}^p
              +C\|u/\delta^t\|_{L^p(\Omega)}^p.
 \label{equ3.11}
\end{align}
Indeed, $\sum_Q\|f\|_{L^q(Q^2)}^p
\le(\sum_Q\|f\|_{L^q(Q^2)}^q)^{p/q}\le C\|f\|_{L^q(\Omega)}^p$.

We then need the following elementary localization inequality that, if $0<\sigma<1$
and $w$ is zero outside $\Omega$, then
\begin{equation}\label{equ3.12}
 [w]_{W^{\sigma,p}(\R^n)}^p
 \le C\sum_Q[w]_{W^{\sigma,p}(Q^1)}^p
              +C\|\delta^{-\sigma}w\|_{L^p(\Omega)}^p.
\end{equation}
To see this, split the pairs $(x,y)\in\Omega\times\Omega$
into the near region
$|x-y|<c\min\{\delta(x),\delta(y)\}$ and its complement,
where $c>0$ depends only on the Whitney dilations.
For sufficiently small fixed $c$, a near pair belongs to $Q^1\times Q^1$
for the Whitney cube containing $x$. For a far pair, the
$1$-Lipschitz property of $\delta$ implies
$|x-y|\ge c'\max\{\delta(x),\delta(y)\}$. Integrating
$2^{p-1}(|w(x)|^p+|w(y)|^p)|x-y|^{-n-\sigma p}$ gives the
weighted term. Pairs crossing the exterior give the same upper
bound, since $|x-y|\ge\delta(x)$ when $y\in\Omega^c$.

When $t\in(0,\,1)$, apply \eqref{equ3.12} with $w=u$ and
$\sigma=t$. By \eqref{equ3.11} and the fact that
$\|u\|_{L^p(\Omega)}\le [\diam(\boz)]^t\|u/\delta^t\|_{L^p(\Omega)}$,
we obtain \eqref{eq:lifting} in this case.

Suppose next that $t\in[1,\,2)$. From \eqref{equ3.11}, it follows that
\begin{equation}\label{equ3.13}
 \|\delta^{1-t}\nabla u\|_{L^p(\Omega)}^p
 +\one_{\{t>1\}}\sum_Q
       [\nabla u]_{W^{t-1,p}(Q^1)}^p
 \le C\left(\|f\|_{L^q(\Omega)}^p
                        +\|u/\delta^t\|_{L^p(\Omega)}^p\right).
\end{equation}
This proves \eqref{eq:lifting} when $t=1$.
For $t\in(1,\,2)$, by \eqref{equ3.13} and \eqref{equ3.12} with $\sigma=t-1$,
we find that
\[
 [\nabla u]_{W^{t-1,p}(\R^n)}^p
 \le C\left(\|f\|_{L^q(\mathbb{R}^n)}^p
                         +\|u/\delta^t\|_{L^p(\Omega)}^p\right).
\]
The lower-order terms follow from the same bounds.
This proves \eqref{eq:lifting} for bounded nonnegative data.

For bounded signed data, write $f=f^+-f^-$, where
$f^+:=\max\{f,0\}$ and $f^-:=\max\{-f,0\}$. Then
$\G f^\pm\le\G|f|$. For general $f\in L^q(\Omega)$, truncate $f^+$ and
$f^-$ separately. The resulting
Green potentials converge in $L^1(\R^n)$, using
$\|\G h\|_{L^1(\Omega)}\le C\|h\|_{L^1(\Omega)}$. Their quotient norms are bounded
by $\|\G|f|/\delta^t\|_{L^p(\Omega)}$. The preceding estimate gives a uniform
$W^{t,p}$ bound, and reflexivity identifies the weak limit with
$\G f$. Weak lower semicontinuity proves \eqref{eq:lifting}.
Thus \eqref{eq:GQ} implies \eqref{eq:WS}. Conversely,
applying \eqref{eq:WS} to $|f|$ and then using
Lemma~\ref{le3.3} bounds the quotient in \eqref{eq:GQ}.
This proves the equivalence and finishes the proof of
Proposition \ref{prop:lifting}.
\end{proof}

\section{Proofs of the main results}\label{s4}
We first prove the volume reverse H\"older estimate and its
higher-order weighted version, then derive the Green quotient
bounds. In the second subsection we combine these bounds with
Proposition~\ref{prop:lifting} to prove Theorems~\ref{t1.1} and
\ref{thm:higher} and Corollary~\ref{c1.1}.

\subsection{Boundary moments and quotient estimates}

\begin{proof}[Proof of Proposition~\ref{prop:RH}]
Let $\beta_\Omega$ be the boundary decay exponent in
Lemma~\ref{le2.9} and $p\in(2,\,p_*)$, where
$p_*:=2+1/(s-\beta_\Omega)$.
Assume that $U$ is admissible on $B(\xi,r)$ in the sense of
Definition~\ref{d1.1}. In particular, its nonnegativity,
weighted integrability, and continuity on the boundary patch allow
the boundary Harnack principle and boundary H\"older estimates below.
Set
\[
 \rho:=r/8,\qquad H:=U/\delta^s,\qquad
 B_U:=U(A_{\xi,\rho})/\rho^s.
\]
As usual, the fixed Lipschitz localization radius $r_0$ may be
decreased by a constant depending only on the domain. We choose it
so that the boundary estimates below apply in every such patch.
If $U(A_{\xi,\rho})=0$, the strong maximum principle gives the
assertion immediately. We therefore assume $B_U>0$.

From Lemma~\ref{le2.9} and interior Harnack principle, it follows that, for any
$x\in D_\rho(\xi)$,
\begin{equation}\label{equ4.1}
 H(x)\leq C B_U\left(\frac{\delta(x)}\rho\right)^{\beta_\Omega-s}.
\end{equation}

Fix an interior point $x_*$ and decrease $r_0$ further so that
$x_*\notin B(\xi,8\rho)$ for all patches under consideration.
By the boundary Harnack comparison \eqref{eq:bhpcomp} in
Lemma~\ref{thm:bhp}, we conclude that, for
any $x\in D_\rho(\xi)$,
\[
 H(x)\leq C\frac{U(A_{\xi,\rho})}{G_s(A_{\xi,\rho},x_*)}
              \frac{G_s(x,x_*)}{\delta(x)^s}.
\]
Denote $S_h:=S_h(\xi,\rho)$ for the distance level set used in
Lemma~\ref{thm:square}. Then, from Lemma \ref{thm:square}, we deduce that,
for any $h\in(0,\,c_0\rho)$,
\[
 \left(\fint_{S_h}
       \left(\frac{G_s(x,x_*)}{h^s}\right)^2
       \,d\mathcal H^{n-1}(x)\right)^{1/2}
 \leq C\frac{G_s(A_{\xi,\rho},x_*)}{\rho^s}.
\]
By this and the surface-measure bound
$\mathcal H^{n-1}(S_h)\leq C\rho^{n-1}$ for
$0<h<c_0\rho$, we find that
\begin{equation}\label{equ4.2}
 \int_{S_h}H^2\,d\mathcal H^{n-1}
 \leq C\rho^{n-1}B_U^2.
\end{equation}

From multiplying \eqref{equ4.2} by the bound for $H^{p-2}$
from \eqref{equ4.1} and using the coarea formula, it follows that
\begin{align}\label{equ4.3}
 \int_{D_\rho(\xi)\cap\{\delta<c_0\rho\}}H^p\,dx
 &\leq C\rho^nB_U^p\int_0^{c_0}
          \tau^{-(s-\beta_\Omega)(p-2)}\,d\tau\leq C\rho^nB_U^p.
\end{align}
On $D_\rho(\xi)\cap\{\delta\geq c_0\rho\}$,
\eqref{equ4.1} gives $H\leq C B_U$. Since
$|D_\rho(\xi)|\sim\rho^n$, combined with \eqref{equ4.3}, we obtain
\[
 \left(\fint_{D_\rho(\xi)}H^p\,dx\right)^{1/p}\leq C B_U.
\]
A corkscrew ball at scale $\rho/2$ lies in $D_\rho(\xi)$.
On this ball, $\delta\sim\rho$. By this and the Harnack inequality,
$U\geq cU(A_{\xi,\rho})$ on the corkscrew ball, and therefore
$\fint_{D_\rho(\xi)}H\,dx\geq c B_U$, which proves
\eqref{eq1.4} for any $p\in(2,\,p_*)$.
H\"older's inequality extends \eqref{eq1.4} to any $p\in(1,\,p_*)$.
Hence
$p_\Omega^{\rm vol}\geq p_*>P_0(n,s,s)$.
This completes the proof of Proposition~\ref{prop:RH}.
\end{proof}

We next extend the boundary moment estimate to the quotient
$U/\delta^t$.
Let $t\in(s,\,\min\{s+1/2,2s\})$ and
\[
 a:=t-s,\qquad b:=s-\beta_\Omega,\qquad \alpha:=2s-t.
\]
Then $0\le a<1/2$, $b>0$, $\alpha>0$, and
$P_{t,\Omega}=(1+2b)/(a+b)>2$.
The universal comparisons in \eqref{eq:Pt} follow from
\[
 P_{t,\Omega}-\frac{2s+1}{t}
 =\frac{\beta_\Omega(2s+1-2t)}{t(t-\beta_\Omega)}>0.
\]
In the planar exceptional case, let
$e:=\beta_\Omega-s+1/2>0$. Then
\[
 P_{t,\Omega}-\frac2{a+1/2}
 =\frac{e(1-2a)}{(a+1/2)(a+1/2-e)}>0.
\]

Let $U$ be admissible on $B(\xi,8\rho)$. Repeating the
boundary Harnack comparison and parallel-surface estimate used
to prove \eqref{equ4.2}, for any $h\in(0,\,c_0\rho)$,
we obtain
\begin{equation}\label{equ4.4}
 \int_{S_h(\xi,\rho)}
       \left(\frac{U}{h^s}\right)^2\dd\mathcal H^{n-1}
 \le C\rho^{n-1}
          \left(\frac{U(A_{\xi,\rho})}{\rho^s}\right)^2.
\end{equation}
Here $S_h(\xi,\rho)$ is the distance level set defined before
Lemma~\ref{thm:square} and the constant $C$ depends only on $n,\,s$,
and $\Omega$.

\begin{lemma}\label{le4.1}
Let $2<P<P_{t,\Omega}$, $\xi\in\partial\Omega$, and
$0<\rho<r_0/8$. If $U$ is admissible on $B(\xi,8\rho)$, then
\begin{equation}\label{equ4.5}
 \left(\int_{D_\rho(\xi)}
           \left(\frac{U}{\delta^t}\right)^P\dd x\right)^{1/P}
 \le C\rho^{n/P-t}U(A_{\xi,\rho}).
\end{equation}
Here the positive constant $C$ depends only on $n$,
$s$, $t$, $P$, $\beta_\Omega$ and $\Omega$, and is independent of
$U,\,\xi$, and $\rho$.
\end{lemma}
\begin{proof}
Set $H:=U/\delta^s$ and
$B_U:=U(A_{\xi,\rho})/\rho^s$. By \eqref{eq2.4},
$H\le C B_U(\delta/\rho)^{-b}$. Since
$U/\delta^t=\delta^{-a}H$, the coarea formula and
\eqref{equ4.4} imply
\begin{align}\label{equ4.6}
 \int_{D_\rho\cap\{\delta<c_0\rho\}}
               \left(\frac{U}{\delta^t}\right)^P\dd x
 &\le C\rho^{n-aP}B_U^P
           \int_0^{c_0}\tau^{-aP-b(P-2)}\dd\tau\\\notag
 &=C\rho^{n-tP}U(A_{\xi,\rho})^P
             \frac{c_0^{\kappa_P}}{\kappa_P},
\end{align}
where
\begin{equation*}
 \kappa_P:=1-aP-b(P-2)=1+2b-(a+b)P>0.
\end{equation*}
On the remainder $\delta\ge c_0\rho$, \eqref{eq2.4} bounds
$U/\delta^t$ by $C\rho^{-t}U(A_{\xi,\rho})$. This, together
with \eqref{equ4.6}, proves \eqref{equ4.5}. This finishes the proof
Lemma \ref{le4.1}.
\end{proof}

The coarea condition is $aP+b(P-2)<1$. At $t=s$ it reduces
to the condition in Proposition~\ref{prop:RH}. At $P=2$ it becomes
$2(t-s)<1$, which explains the upper threshold $s+1/2$.

We pass from the boundary moment to an annular kernel bound and
then to the quotient estimate for general data.
Set $\mathcal K_t(x,y):=G_s(x,y)/\delta(x)^t$ for $x,y\in\Omega$,
and extend this kernel by zero outside $\Omega\times\Omega$.
\begin{lemma}\label{le4.2}
For any $P\in(2,\,P_{t,\Omega})$, $y\in\R^n$, and $r\in(0,\,\fz)$,
\begin{equation}\label{equ4.7}
 \|\mathcal K_t(\cdot,y)\|_{L^P(\{r<|x-y|\le2r\})}
 \le Cr^{\alpha-n+n/P}.
\end{equation}
Here the positive constant $C$ depends only on $n,\,s,
\,t,\,P,\,\beta_\Omega$ and $\Omega$.
\end{lemma}
\begin{proof}
The kernel vanishes when
$y\notin\Omega$, and the annulus misses $\Omega$ when $r\ge D$.
Thus let $y\in\Omega$ and $0<r<D$. Choose a fixed
$\theta\in(0,1/16)$ so small that $8\theta D<r_0$, where $r_0$
is the boundary localization radius, and let $\rho:=\theta r$.
Where $\delta(x)\ge\rho/8$, the fundamental-solution
bound $G_s(x,y)\le C|x-y|^{2s-n}$ gives
$\mathcal K_t(x,y)\le Cr^{2s-t-n}=Cr^{\alpha-n}$.
This part contributes at most $Cr^{\alpha-n+n/P}$.

Cover the remaining part by boundary balls $B(\xi_j,\rho)$, with
at most $C\theta^{-n}$ balls and $|\xi_j-y|\ge r/2>8\rho$.
To obtain this cover, take a maximal $\rho/4$-separated family
$\{\xi_j\}$ of boundary points with
$r/2\le|\xi_j-y|\le3r$. Volume packing bounds its cardinality by
$C\theta^{-n}$, independently of $r$ and $y$. By decreasing
$\theta$ if necessary, each enlarged patch stays in a boundary
coordinate neighborhood. In each patch
$U=G_s(\cdot,y)$ satisfies the hypotheses of Lemma~\ref{le4.1}.
Furthermore, by separation from
the pole, $U(A_{\xi_j,\rho})\le Cr^{2s-n}$. Therefore
\[
 \int_{D_\rho(\xi_j)}\mathcal K_t(x,y)^P\dd x
 \le C\rho^{n-tP}r^{(2s-n)P}
 \le Cr^{n+(\alpha-n)P}.
\]
Since the number of patches is independent of $r$ and $y$, summing
the estimates for the interior and boundary parts proves
\eqref{equ4.7} at any scale $r\in(0,\,D)$. This finishes
the proof of Lemma \ref{le4.2}.
\end{proof}

\begin{lemma}\label{le4.3}
Let $n\ge2$, $s\in(0,1)$, and $\Omega\subset\R^n$ be a bounded
Lipschitz domain. Suppose that $t\in(s,\,\min\{s+1/2,2s\})$,
$p\in(1,\,P_{t,\Omega})$, $q\in(1,\,\infty)$, and \eqref{e1.9}
holds. Then, for any $f\in L^q(\Omega)$,
\begin{equation}\label{equ4.8}
 \|\G|f|/\delta^t\|_{L^p(\Omega)}\le C\|f\|_{L^q(\Omega)}.
\end{equation}
Here the positive constant $C$ depends only on $n$,
$s$, $t$, $p$, $q$, and $\Omega$.
\end{lemma}

\begin{proof}
On the exact gain line $1/p=1/q-\alpha/n$, choose
$P\in(\max\{2,p\},\,P_{t,\Omega})$. Lemma~\ref{le4.2}
verifies \eqref{eq:abstractannular} for $K=\mathcal K_t$.
By this and Lemma~\ref{le2.7} with $\alpha=2s-t$, we obtain
\eqref{equ4.8}.

For the remaining exponents, define
$\lambda_\alpha:=n/(n-\alpha)<2$. The strict inequality follows
from $0<\alpha<n/2$.
If $p>\lambda_\alpha$, then take $q_0=np/(n+\alpha p)>1$.
\eqref{e1.9} gives $q\ge q_0$. Then the inclusion
$L^q(\Omega)\subset L^{q_0}(\Omega)$ yields \eqref{equ4.8}. If
$p\in(1,\,\lambda_\alpha]$, choose $1<q_0<q$ sufficiently close to $1$ that
$p_0=nq_0/(n-\alpha q_0)<P_{t,\Omega}$. Then $p_0>\lambda_\alpha\ge p$.
From applying the gain line estimate at $(q_0,p_0)$ and lower the output
Lebesgue exponent, we deduce that \eqref{equ4.8} holds true. This finishes the proof
of Lemma \ref{le4.3}.
\end{proof}

\subsection{Proofs of Theorems \ref{t1.1} and \ref{thm:higher} and Corollary \ref{c1.1}}

\begin{proof}[Proof of Theorem~\ref{t1.1}]
Define $\lambda_s:=n/(n-s)$ and
$q_{\rm e}:=2n/(n+2s)$. 
Since $n>2s$, we have $1<\lambda_s<2$ and $1<q_{\rm e}<n/s$.

We first consider the case that $p\in(\lambda_s,\,p_\Omega^{\rm vol})$.
Set $q_p=np/(n+sp)$. Then $1<q_p<n/s$ and
$1/p=1/q_p-s/n$.
Choose $P\in(\max\{p,2\},\,p_\Omega^{\rm vol})$.
By Lemmas~\ref{le2.6} and \ref{le2.7}, applied to the
kernel $K=\mathcal K_s$ with $\alpha=s$, annular exponent
$P$, and target exponent $p$, we find that
\begin{equation*}
 \|\G|f|/\delta^s\|_{L^p(\Omega)}
 \leq C\|f\|_{L^{q_p}(\Omega)}.
\end{equation*}
From this and Proposition~\ref{prop:lifting}, applied with $t=s$ and
$\gamma=s$, it follows that
\[
 \|\G f\|_{W^{s,p}(\R^n)}
 \leq C\|f\|_{L^{q_p}(\Omega)}.
\]
The index condition \eqref{eq1.5} is equivalent to $q\geq q_p$.
The fact that $\|f\|_{q_p}\leq|\Omega|^{1/q_p-1/q}\|f\|_q$
yields \eqref{eq1.9} in this range.

We next consider $1<p\le\lambda_s$.
Choose $q_0$ and $p_0$ so that
\[
 1<q_0<\min\{q,q_{\rm e}\},\qquad
 p_0=\frac{nq_0}{n-sq_0}.
\]
Then $\lambda_s<p_0<2<p_\Omega^{\rm vol}$. By the estimate already
proved, we find that
\begin{equation}\label{equ4.9}
 \|\G|f|/\delta^s\|_{L^p(\Omega)}
 \leq |\Omega|^{1/p-1/p_0}
       \|\G|f|/\delta^s\|_{L^{p_0}(\Omega)}
 \leq C\|f\|_{L^q(\Omega)}.
\end{equation}
Choose $r$ so that $1<r<\min\{q,p\}$. Then
\[
 0<\gamma_r=n(1/r-1/p)<n(1-1/p)\leq s.
\]
Apply Proposition~\ref{prop:lifting} with $t=s$, data exponent $r$
and target exponent $p$. This, together with \eqref{equ4.9}
and $L^q(\Omega)\subset L^r(\Omega)$, implies
\[
 \|\G f\|_{W^{s,p}(\R^n)}
 \leq C\big(\|f\|_{L^r(\Omega)}
             +\|\G|f|/\delta^s\|_{L^p(\Omega)}\big)
 \leq C\|f\|_{L^q(\Omega)}.
\]
Thus \eqref{eq1.9} holds for all the stated indices.

It remains to verify the weak identity \eqref{eq1.3}.
Let $f\in L^q(\Omega)$ with $1<q<\infty$. For each integer
$j\ge1$, define the bounded truncation
$f_j:=\max\{-j,\min\{f,j\}\}$, and put $u_j:=\G f_j$ and
$u:=\G f$. Then $f_j\to f$ in $L^q(\Omega)$. By
\eqref{eq1.9}, applied to $f_j-f$, we conclude that
\[
 \|u_j-u\|_{W^{s,p}(\R^n)}
 \leq C\|f_j-f\|_{L^q(\Omega)}\longrightarrow0.
\]
For every $\phi\in C_{\rm c}^\infty(\Omega)$, the solution $u_j$ satisfies
\[
 \mathcal E_s(u_j,\phi)
 =\int_\Omega f_j\phi\,dx.
\]
From H\"older's inequality on $\R^n\times\R^n$, the absolute
value of the energy form with $u_j-u$ in place of $u_j$ is at most
\[
 \frac{c_{(n,s)}}2
 [u_j-u]_{W^{s,p}(\R^n)}[\phi]_{W^{s,p'}(\R^n)},
\]
which tends to zero. The right-hand side converges by
$L^q$ convergence and $\phi\in L^{q'}(\Omega)$. Therefore, 
for any function $\phi\in C_{\rm c}^\infty(\Omega)$,
\begin{equation}\label{equ4.10}
 \mathcal E_s(u,\phi)
 =\int_\Omega f\phi\,dx.
\end{equation}
The form $\mathcal E_s(u,\phi)$ is absolutely convergent, and
$u=0$ almost everywhere in $\Omega^c$ by its Green representation.
Since $\Omega$ is bounded, $u\in W^{s,p}(\R^n)$ satisfying that $u=0$ on $\boz^c$ also belongs to
$L^1(\R^n)$. By Remark~\ref{re1.1}, \eqref{equ4.10} is equivalent to
\eqref{eq1.3}.
This completes the proof of Theorem \ref{t1.1}.
\end{proof}

\begin{proof}[Proof of Theorem~\ref{thm:higher}]
Let $\alpha:=2s-t$. Lemma~\ref{le4.3} gives \eqref{eq:GQ}.
If $q<p$, apply this and Proposition~\ref{prop:lifting} directly. If $q\ge p$,
choose $q_0\in(1,p)$ sufficiently close to $p$ so that
$0<n(1/q_0-1/p)<\alpha$, and use
$\|f\|_{L^{q_0}(\Omega)}\le |\Omega|^{1/q_0-1/q}\|f\|_{L^q(\Omega)}$.
The quotient remains in the same target $L^p$ space. 
Proposition~\ref{prop:lifting} proves the full norm estimate.
To verify the weak equation, first take bounded truncations
$f_j\to f$ in $L^q(\Omega)$ when $q<\infty$. The preceding
estimate gives $\G f_j\to\G f$ in $W^{t,p}(\R^n)$, hence also
in $W^{s,p}(\R^n)$ because $t\ge s$. By the bounded-data Green
identity and Fubini's theorem, each $\G f_j$ satisfies the
weak identity against $\phi\in C_{\rm c}^\infty(\Omega)$.
H\"older's inequality permits passage to the limit in the energy
form, while $L^q$ convergence permits passage to the limit on the
right-hand side. This proves the asserted weak equation and finishes the proof
of Theorem~\ref{thm:higher}.
\end{proof}

\begin{proof}[Proof of Corollary~\ref{c1.1}]
Let $\alpha:=2s-t$ and
\[
 Q_*:=\frac{nP_{t,\Omega}}{n+\alpha P_{t,\Omega}}.
\]
Since $z\mapsto nz/(n+\alpha z)$ is strictly increasing and
$P_{t,\Omega}>P_0(n,s,t)>2$, we have
$1<Q_0(n,s,t)<Q_*<n/\alpha$. Choose
$\eta:=(Q_*-Q_0(n,s,t))/2>0$. For
$q\in(1,\,Q_0(n,s,t)+\eta)$, the exponent
$p_q=nq/(n-\alpha q)$ satisfies $p_q<P_{t,\Omega}$.
Since $\frac1{p_q}=\frac1q-\frac{2s-t}{n}$,
Theorem~\ref{thm:higher} applies. At $t=s$, this also proves
\eqref{eq1.9} and finishes the proof
of Corollary~\ref{c1.1}.
\end{proof}

\section{Optimality of the universal exponents}\label{s5}
We prove Theorem~\ref{thm:optimal} with $s$ and $t$ fixed.
For each domain $O\subset\R^n$, write $\mathcal G_{s,O}$ for
its Green operator and define the boundary distance by
$\delta_O(x):=\dist(x,\partial O)$, for any $x\in O$. If
$n>1+2s$, we construct a Lipschitz graph above a regular Cantor
set of dimension slightly below $n-1-2s$. If $n=2$ and $s\ge1/2$,
we instead use homogeneous $s$-harmonic function in sectors. In both cases,
Lemma~\ref{le3.3} converts divergence of $u/\delta^t$
into failure of global $W^{t,p}$ regularity, including $t=1$
and $t>1$.

\subsection{The Cantor-set counterexample}
Throughout this subsection, assume $n>1+2s$ and
$0<d<n-1-2s$. Set
\begin{equation*}
 \kappa:=n-1-2s-d>0.
\end{equation*}
Choose a compact $d$-Ahlfors regular set
$E\subset(-1/8,1/8)^{n-1}$. By this we mean that there are a
finite positive measure $\mu$ supported on $E$, constants
$c_E,C_E\in(0,\,\fz)$, and a radius $r_E\in(0,\,\fz)$ such that, for any
$e\in E$ and $r\in(0,\,r_E)$,
\[
 c_E r^d\le\mu(B(e,r))\le C_E r^d.
\]
Such sets exist for every stated $d$. If $d$ is not an integer, write $d=m+\theta$, where
$m:=\lfloor d\rfloor$ the greatest integer less than $d$ and $\theta\in(0,\,1)$. Take the product of
an $m$-dimensional cube with the two-branch Cantor set of
contraction ratio $2^{-1/\theta}$, add zero coordinates to
embed it in $\R^{n-1}$, and rescale. For integer $d$, use a $d$-dimensional
cube. The product measure has the displayed bounds.
Set $F:=E\times(-\infty,0]$. For any $e\in E$, write
$\xi_e:=(e,0)$.
The set $F$ is used only for estimates and is not the exterior of
the final Lipschitz domain.

For $h>0$, define the open $h$-neighborhood by setting
$$F_h:=\{x\in\R^n:\dist(x,F)<h\}$$
and
$$E_h:=\{x'\in\R^{n-1}:\dist(x',E)<h\}.$$

\begin{lemma}[Covering and volume estimates]\label{le5.1}
There is a positive constant $C$, depending only on
$n,\,d,\,c_E$, and $C_E$, such that the following assertions hold.
\begin{itemize}
 \item[\rm (i)] For any
$e\in E$ and $0<\rho<R<r_E/10$, the set
$F\cap B(\xi_e,R)$ is covered by at most
$C(R/\rho)^{d+1}$ balls of radius $\rho$ centered on $F$.
\item[\rm (ii)]For any
$e\in E$ and $0<h<R<r_E/20$,
\begin{equation}\label{equ5.1}
 |F_h\cap B(\xi_e,R)|\le C R^{d+1}h^{n-d-1}.
\end{equation}

\item[\rm (iii)]For any $0<h<r_E/20$,
\begin{equation}\label{equ5.2}
 c h^{n-1-d}\le |E_h|\le C h^{n-1-d},
\end{equation}
where the constants $c$ and $C$ depends on $E$.
\end{itemize}
\end{lemma}

\begin{proof}
From the lower measure bounds on balls and the upper bound on $B(e,3R)$,
it follows that a maximal $\rho$-separated family in $E\cap B(e,2R)$ has at most
$C(R/\rho)^d$ members.
Taking the product of this covering with a covering of
$[-R,0]$ by at most $C R/\rho$ intervals gives (i) after a fixed enlargement of the balls. 
Covering $F\cap B(\xi_e,2R)$ with balls with radius $h$
and enlarging the balls proves \eqref{equ5.1}.
The same argument applied to all of $E$ yields the upper estimate
in \eqref{equ5.2}. For the lower one, a maximal $h$-separated
family must have at least $c h^{-d}$ members, since its enlarged
balls cover $E$ and each has measure at most $C_Eh^d$.
The disjoint Euclidean balls of radius $h/3$ around those members
are contained in $E_h$ and have total volume at least
$c h^{-d}h^{n-1}$. This finishes the proof of Lemma \ref{le5.1}.
\end{proof}

Let $I_{2s}$ denote the Riesz potential, normalized so that
$(-\Delta)^sI_{2s}\psi=\psi$ for smooth compactly supported
$\psi$. Assume that
$0\le\psi\in C_{\rm c}^\infty(B_2)$ with $\psi\ge1$ on $B_1$.
For a center $z\in\R^n$ and a radius $\rho>0$, take a
sufficiently large positive constant $A$ and define the
single-ball barrier by
\begin{equation}\label{equ5.3}
 W_{z,\rho}(x):=A\rho^{-2s}
 I_{2s}\big[\psi((\cdot-z)/\rho)\big](x).
\end{equation}
By positivity of the Riesz kernel, we find that, for any $x\in\R^n$,
\begin{equation*}
 W_{z,\rho}(x)\ge0.
\end{equation*}
Scaling and the choice of $A$ give, for any $x\in B(z,\rho)$,
\begin{equation}\label{equ5.4}
 W_{z,\rho}(x)\ge1.
\end{equation}
From the normalization of $I_{2s}$, it follows that, for any $x\in\R^n$,
\begin{equation}\label{equ5.5}
 (-\Delta)^sW_{z,\rho}(x)
       =A\rho^{-2s}\psi((x-z)/\rho)\ge0.
\end{equation}
Finally, integration of the Riesz kernel yields, whenever
$|x-z|\ge4\rho$,
\begin{equation}\label{equ5.6}
 W_{z,\rho}(x)\le C
       \left(\frac{\rho}{|x-z|}\right)^{n-2s}.
\end{equation}
Thus, each barrier is smooth and bounded, with an integrable fractional
tail. 

We now construct a fixed Lipschitz graph.
For a parameter $M\ge1$ and every $x'\in\R^{n-1}$,
define the graph function by
\begin{equation}
 \phi_M(x'):=\max\{-1,-M\dist(x',E)\}.
\end{equation}
The associated bounded epigraph domain is
\begin{equation}
 \Omega_M:=\{(x',x_n):x'\in(-2,2)^{n-1},\
                    \phi_M(x')<x_n<1\}.
\end{equation}
The graph is $M$-Lipschitz. Near the lateral sides of the box its
value is $-1$, so the remaining boundary consists of ordinary
box corners. Thus $\Omega_M$ is a bounded connected Lipschitz domain.
In particular $(e,0)\in\partial\Omega_M$ for every $e\in E$.
Let $u_M:=\mathcal G_{s,\Omega_M}1$. Then
$u_M$ is bounded, nonnegative, continuous and strictly positive in
the interior.

\begin{lemma}[Slow decay]\label{le5.2}
For any $\nu\in(0,\,\fz)$, there exist positive constants $M\in[1,\,\fz)$,
$K\in[4,\,\fz)$, $\varepsilon\in(0,1/4)$, and
$r_{\rm init},\,c\in(0,\,\fz)$, depending only on $n,\,s,\,\nu,$ and $E$, such that, for any $j\ge0$ and any
$x:=(x',x_n)$ satisfying $\dist(x',E)<r_j:=r_{\rm init}(2K)^{-j}$ and
$r_j\le x_n\le2r_j$,
\begin{equation}\label{eq5.9}
 u_M(x)\ge c r_j^\nu.
\end{equation}
\end{lemma}
\begin{proof}
Choose $\eta\in(0,1/4)$ so small that
\begin{equation}\label{eq5.10}
 -\log(1-2\eta)/\log8<\nu.
\end{equation}
For any $K\in[4,\,\fz)$, $\varepsilon\in(0,1/4)$, $e\in E$,
and $r\in(0,\,\fz)$, define the good annulus by
\[
 \mathcal R(e,r):=\{y:r\le|y-\xi_e|\le Kr,
                         \dist(y,F)\ge\varepsilon r\}.
\]
Take $M\in(K/\varepsilon,\,\fz)$ and $r\in(0,\,\frac1{4K})$, these sets are contained in $\Omega_M$.
Indeed, for a point $y\in\Omega_M^c$,
$y_n\le0$ and $\dist(y',E)\le |y_n|/M\le Kr/M<\varepsilon r$.

\emph{Step 1: the Poisson mass of a good annulus.}
Let $R\in(0,r_E/20)$ such that $KR$ is sufficiently small compared with
$r_E$ and $B:=B(\xi_e,R)$, and let $h_{e,R}$ be the
$s$-harmonic function in $B$ with exterior data
$\one_{\mathcal R(e,R)}$. By the Poisson representation for a ball
(see, for instance, \cite[Section~2, p.~3]{M06}), for
any $x\in B$ and $y\notin\overline B$, its kernel is
\[
 P_B(x,y)=\frac{\Gamma(n/2)\sin(\pi s)}{\pi^{n/2+1}}
 \frac{(R^2-|x-\xi_e|^2)^s}
      {(|y-\xi_e|^2-R^2)^s|x-y|^n}.
\]
Let $\tau\in(0,\,K-1)$. Cover $B^c\setminus\mathcal R(e,R)$ by
the far region $|y-\xi_e|>KR$, the thin shell
$R<|y-\xi_e|<(1+\tau)R$, and the remaining part of $F_{\varepsilon R}$.
For any $x$ satisfying $|x-\xi_e|\le R/2$,
the decay of the Poisson kernel provides the far-region estimate
\begin{equation}\label{eq5.11}
 \int_{|y-\xi_e|>KR}P_B(x,y)\dd y\le C K^{-2s},
\end{equation}
integrating the kernel across the thin shell yields
\begin{equation}\label{eq5.12}
 \int_{R<|y-\xi_e|<(1+\tau)R}P_B(x,y)\dd y
       \le C\tau^{1-s},
\end{equation}
and the kernel bound and \eqref{equ5.1} imply
\begin{equation}\label{eq5.13}
 \int_{\substack{(1+\tau)R\le|y-\xi_e|\le KR\\y\in F_{\varepsilon R}}}
 P_B(x,y)\dd y
       \le C\tau^{-s}K^{d+1}\varepsilon^{n-d-1}.
\end{equation}
Here the positive constant $C$ depends only on $n,\,s,\,d$, and $E$. Choose $K$ large in
\eqref{eq5.11}, then $\tau$ small in
\eqref{eq5.12}, and finally $\varepsilon$ small in
\eqref{eq5.13}. Then each right-hand side in \eqref{eq5.11}, \eqref{eq5.12}, and
\eqref{eq5.13} is at most
$\eta/3$. Since the full
Poisson mass is 1,
for any $x\in B(\xi_e,R/2)$, we obtain
\begin{equation}\label{eq5.14}
 h_{e,R}(x)\ge1-\eta.
\end{equation}

\emph{Step 2: a small barrier for the actual exterior inside the ball.}
Let $\Lambda:=2K$, and choose $r_{\rm init}$ so small that
$Kr_{\rm init}<\min\{1/4,r_E/100\}$. If
$R=\Lambda r\le r_{\rm init}$, then
\[
 \Omega_M^c\cap B(\xi_e,R)
 \subset\{y:y_n\le0,\ \dist(y',E)\le R/M\}.
\]
Take $\rho:=R/M$. For any
$y=(y',y_n)\in\Omega_M^c\cap B(\xi_e,R)$, choose $e_y\in E$ with
$|y'-e_y|=\dist(y',E)$. Then $(e_y,y_n)\in F$ satisfies
$|y-(e_y,y_n)|\le\rho$ and belongs to
$F\cap B(\xi_e,2R)$. By Lemma~\ref{le5.1}, $F\cap B(\xi_e,2R)$
is covered by at most $C(R/\rho)^{d+1}=CM^{d+1}$ balls
$B(z_i,\rho)$ with centers $z_i\in F$. Enlarging these balls to
$B(z_i,2\rho)$ therefore covers $\Omega_M^c\cap B(\xi_e,R)$.
Choose one such covering for each $e$, $R$, and $M$. Using the
single-ball barriers in \eqref{equ5.3}, define the
corresponding aggregate barrier by
\[
 \mathcal W_{e,R}^{M}:=\sum_i W_{z_i,2\rho}.
\]
From \eqref{equ5.4}, it follows that $\mathcal W_{e,R}^{M}\ge1$
on $\Omega_M^c\cap B(\xi_e,R)$. By \eqref{equ5.5}, we have that
 $(-\Delta)^s\mathcal W_{e,R}^{M}\ge0$ globally. If
$x\in\mathcal R(e,r)$, then, for any center $z_i$,
$|x-z_i|\ge\dist(x,F)\ge\varepsilon r$.
Taking $M$ large enough that $8R/M<\varepsilon r$,
\eqref{equ5.6} implies that, for any $x\in\mathcal R(e,r)$,
\begin{equation}\label{eq5.15}
 \mathcal W_{e,R}^{M}(x)\le C M^{d+1}
                \left(\frac{R/M}{\varepsilon r}\right)^{n-2s}.
\end{equation}
Using \eqref{eq5.15} together with $R:=\Lambda r$
and $\kappa:=n-1-2s-d$, we obtain
\begin{equation}\label{eq5.16}
 \mathcal W_{e,R}^{M}(x)
       \le C(\Lambda/\varepsilon)^{n-2s}M^{-\kappa}.
\end{equation}
Since $\kappa>0$, we can choose $M\in(K/\varepsilon,\,\fz)$ large enough such that, for any $x\in\mathcal R(e,r)$, $\mathcal W_{e,R}^{M}(x)\le\eta$. The choice is
independent of $e$ and $r$.

\emph{Step 3: iteration of the comparison estimate.}
Fix this $M$ and abbreviate $u:=u_M$. Let
$r_j:=r_{\rm init}\Lambda^{-j}$
and
\[
 m_j:=\inf\{u(x):e\in E,\ x\in\mathcal R(e,r_j)\}.
\]
By the facts that $\mathcal R(e,r_0)$ is  compact and $\mathcal R(e,r_0)\subset\Omega_M$, we find that $m_0>0$. For $j\ge1$, set $r:=r_j$ and
$R:=r_{j-1}$. Comparing $u$ with
\[
 v:=m_{j-1}(h_{e,R}-\mathcal W_{e,R}^{M}),
\]
we obtain that $(-\Delta)^sv\le0<1=(-\Delta)^su$ in $\Omega_M\cap B(\xi_e,R)$ and
$v\le m_{j-1}\one_{\mathcal R(e,R)}\le u$ outside the ball $B(\xi_e,R)$. In the part of
$\Omega_M^c$ inside the ball $B(\xi_e,R)$, $h_{e,R}\le1\le\mathcal W_{e,R}^{M}$, so
$v\le0=u$. The comparison principle therefore yields $u\ge v$ on $\mathbb{R}^n$.

If $x\in\mathcal R(e,r)$, then $|x-\xi_e|\le Kr=R/2$.
Equations \eqref{eq5.14} and \eqref{eq5.16} imply
\begin{equation}\label{eq5.17}
 m_j\ge(1-2\eta)m_{j-1}.
\end{equation}
By iterating \eqref{eq5.17}, we find that
\begin{equation}\label{eq5.18}
 m_j\ge m_0(1-2\eta)^j.
\end{equation}
Define the decay exponent
$\nu_0:=-\log(1-2\eta)/\log\Lambda$. Since
$\Lambda=2K\ge8$, \eqref{eq5.10} gives $\nu_0<\nu$.
Since $r_j=r_{\rm init}\Lambda^{-j}$, \eqref{eq5.18} implies
$m_j\ge c r_j^{\nu_0}\ge c' r_j^\nu$ after adjusting the
constant for $r_{\rm init}$.
Finally, if $\dist(x',E)<r_j$ and $r_j\le x_n\le2r_j$,
choose a nearest $e\in E$. Then
$r_j\le|x-\xi_e|\le\sqrt5r_j<Kr_j$ and
$\dist(x,F)\ge r_j>\varepsilon r_j$. Thus
$x\in\mathcal R(e,r_j)$, which proves \eqref{eq5.9}.
This finishes the proof of Lemma \ref{le5.2}.
\end{proof}

We can now choose the dimension of $E$ and the decay exponent
to obtain the required failure of Sobolev regularity.
\begin{proof}[Proof of Theorem~\ref{thm:optimal} when $n>1+2s$]
Let $p\in((2s+1)/t,\,\fz)$ and define the positive gap
$\iota:=tp-(1+2s)$. Choose
$\kappa\in(0,\,\min\{n-1-2s,\iota/4\})$. Let
$d:=n-1-2s-\kappa>0$ and $\nu:=\iota/(4p)$.
Then
\begin{equation*}
 n-d-p(t-\nu)=\kappa-3\iota/4<0.
\end{equation*}
Choose the regular set $E$ and the domain $\Omega_M$ as above,
and let $u:=\mathcal G_{s,\Omega_M}1$. For any $j\ge0$, define
\[
 T_j:=\{(x',x_n):\dist(x',E)<r_j,\ r_j<x_n<2r_j\}.
\]
These sets are pairwise disjoint and contained in $\Omega_M$.
By \eqref{equ5.2}, their volumes satisfy
\begin{equation}\label{eq5.19}
 |T_j|\ge c r_j^{n-d}.
\end{equation}
Then, from Lemma~\ref{le5.2}, it follows that, for any $x\in T_j$,
\begin{equation}\label{eq5.20}
 u(x)\ge c r_j^\nu.
\end{equation}
Here the positive constant $c$ may depend on $E,n,s,p,t,M$, but is
independent of $j$. Moreover,
for any $x\in T_j$,
\begin{equation}\label{eq5.21}
 \delta_{\Omega_M}(x)\le\sqrt5r_j.
\end{equation}
 Combining \eqref{eq5.19},
\eqref{eq5.20}, and \eqref{eq5.21}, we obtain
\begin{equation}\label{eq5.22}
 \int_{\Omega_M}\left|\frac{u(x)}{\delta_{\Omega_M}(x)^t}\right|^pdx
 \ge c\sum_{j=0}^\infty r_j^{n-d-p(t-\nu)}=\infty.
\end{equation}
If $u$ belonged to $W^{t,p}(\R^n)$, Lemma~\ref{le3.3}
would contradict \eqref{eq5.22}. This proves the
claim with right-hand side exactly one. The construction covers
all $n\ge3$, $s\in(0,1)$, and also $n=2$, $s<1/2$. This finishes the
proof of Theorem~\ref{thm:optimal} when $n>1+2s$.
\end{proof}

\subsection{The planar sector counterexample}

\begin{proof}[Proof of Theorem~\ref{thm:optimal} when $n=2$ and $s\in[1/2,\,1)$]
Let $\mathcal C_\Theta$ be the sector with half-opening
$\Theta\in(\pi/2,\pi)$, and let $U_\Theta$ be a positive
$s$-harmonic function on sector $\mathcal C_\Theta$ satisfying $U_\Theta\equiv0$ outside the sector. 
Denote its homogeneity degree by
$\gamma_\Theta:=\gamma_{2,s}(\Theta)\in(0,s)$.
By the sector homogeneity theorem (see, for instance,
\cite[Theorem~3.2]{M06}), as $\Theta\to\pi^-$,
\begin{equation}\label{eq5.23}
 \gamma_{2,s}(\Theta)\longrightarrow s-\frac12.
\end{equation}
By \eqref{eq5.23}, as $\Theta\to\pi^-$,
\begin{equation}\label{eq5.24}
 \frac2{t-\gamma_{2,s}(\Theta)}
 \longrightarrow\frac4{1+2(t-s)}.
\end{equation}
At $s=1/2$, the limit in \eqref{eq5.24}
equals $(2s+1)/t$. By \eqref{eq5.24}, for any
$p>P_0(2,s,t)$, we can choose a proper
sector for which $p(t-\gamma_\Theta)>2$.

Let $\chi\in C_{\rm c}^\infty(B_2)$ be nonnegative and radial satisfying $\chi=1$ on
$B_1$. Set $\Omega:=\mathcal C_\Theta\cap B_4$ and
$V:=\chi U_\Theta$.
The domain $\Omega$ is bounded and Lipschitz. From \cite[formula~(2-2)]{TTV2018},
it follows that, for any $x\in\mathcal C_\Theta$,
\[
 U_\Theta(x)\sim |x|^{\gamma_\Theta-s}
                 \dist(x,\mathcal C_\Theta^c)^s.
\]
The function $U_\Theta$ is
$C^s$ on closed annuli, including across the lateral boundary.
Indeed, its angular trace is $C^s$ on the unit circle, and
homogeneity transfers this regularity to annuli away from the
vertex (see, for instance, \cite[Section~5]{TTV2018}). By the energy estimate (see, for instance,
\cite[Proposition~2.7]{TTV2018}), $V\in W^{s,2}(\R^2)$.
We verify that $f:=(-\Delta)^sV|_\Omega$ is bounded.
From the harmonicity of $U_\Theta$, it follows that,
for $x\in \mathcal C_\Theta\cap B_{1/2}$, 
\[
 f(x)=c_{(2,s)}\int_{\R^2}
       \frac{(1-\chi(y))U_\Theta(y)}{|x-y|^{2+2s}}\,dy.
\]
Since $\gamma_\Theta<2s$, the integral converges, and
\[
 |f(x)|\le C\left(1+\int_2^\infty r^{\gamma_\Theta-2s-1}\,dr\right)
 \le C.
\]
On the remaining annulus $1/2\le|x|\le4$,
\begin{align*}
 f(x)={}&U_\Theta(x)(-\Delta)^s\chi(x)\\
 &-c_{(2,s)}\int_{\R^2}
 \frac{(\chi(x)-\chi(y))(U_\Theta(x)-U_\Theta(y))}
 {|x-y|^{2+2s}}\,dy.
\end{align*}
Near the diagonal, by the Lipschitz bound for $\chi$ and the annular
$C^s$ bound for $U_\Theta$, we find that
\[
 \int_{|x-y|<1/8}
 \frac{|\chi(x)-\chi(y)|\,|U_\Theta(x)-U_\Theta(y)|}
 {|x-y|^{2+2s}}\,dy
 \le C\int_0^{1/8}r^{-s}\,dr<\infty.
\]
The separated part is bounded by the same integrable tail as above.
Combining the estimates near the vertex and on the annulus gives
$\|f\|_{L^\infty(\Omega)}\le C$, where $C$ depends only on
$s,\Theta$ and the fixed cutoff $\chi$. From the uniqueness of the energy
solution and the positivity of the Green function, we obtain
\[
 V=\mathcal G_{s,\Omega}f.
\]
Moreover, for any $x\in\Omega$,
\[
 0\le V(x)\le\|f\|_\infty(\mathcal G_{s,\Omega}1)(x).
\]
Choose a closed angular subinterval $J$ strictly inside the sector.
For $x=r\omega$ with $\omega\in J$ and $0<r<1/2$, we have
$V(x)\ge cr^{\gamma_\Theta}$ and $\delta(x)\sim r$. Hence
\[
 \int_\Omega |V/\delta^t|^pdx
 \ge c\int_0^{1/2}r^{1-p(t-\gamma_\Theta)}\,dr=\infty.
\]
Lemma~\ref{le3.3}
therefore yields $\mathcal G_{s,\Omega}1\notin W^{t,p}(\R^2)$.
This completes the proof of Theorem~\ref{thm:optimal}.
\end{proof}

\medskip

\noindent\textbf{Data Availability}\quad  Data sharing not applicable
to this article as no datasets were generated or analyzed during
the current study.

\section*{Declarations}

\noindent\textbf{Conflict of interest}\quad The authors declare
that they have no conflict of interest.

\medskip

\noindent\textbf{Declaration of AI Use}\quad During the preparation of this manuscript, the authors used OpenAI GPT-5.6 for finding the counterexample using regular Cantor sets on the boundary, English language editing, and stylistic improvements, including grammar, wording, and readability. All mathematical content was developed and verified by
the authors, who take full responsibility for the final manuscript.

\bigskip

\noindent Wenxian Ma and Sibei Yang

\medskip

\noindent School of Mathematics and Statistics, Gansu Key Laboratory of Applied Mathematics
and Complex Systems, Lanzhou University, Lanzhou 730000, People's Republic of China

\smallskip

\smallskip

\noindent {\it E-mails}: \texttt{mawenxian@lzu.edu.cn} (W. Ma)

\hspace{0.888cm} \texttt{yangsb@lzu.edu.cn} (S. Yang)

\end{document}